\documentclass[12pt]{article}
\usepackage{amsmath,amssymb,amsthm,amscd,latexsym,}
 
\usepackage{graphicx}

\makeatletter
\renewcommand{\mod}[1]{\allowbreak \if@display \mkern 8mu \else
\mkern 5mu\fi {\operator@font mod}\,\,#1}
\makeatother

\newcommand{\bc}{\mathbb C}

 \newcommand{\bq}{\mathbb Q}
\newcommand{\br}{\mathbb R}
\newcommand{\bz}{\mathbb Z}

\newcommand{\bp}{\mathbb P}

\newcommand{\aaa}{\mathbb A}

\DeclareMathOperator{\rk}{rk}

\newtheorem{theorem}{Theorem}
\newtheorem{proposition}[theorem]{Proposition}

\newcommand\SSS{\mathfrak S}
\newcommand\AAA{\mathfrak A}

\begin{document}
\title{Degenerations of K\"ahlerian K3 surfaces with finite
symplectic automorphism groups, II\footnote{With support by Russian
Sientific Fund N 14-50-00005.}}
\date{}
\author{Viacheslav V. Nikulin}
\maketitle

\begin{abstract}
We prove the main Conjecture 4 of our paper \cite{Nik9}.
Further, we apply these results to classification of
degenerations of codimension one of K\"ahlerian K3 surfaces with finite
symplectic automorphism groups.
\end{abstract}

\centerline{Dedicated to V.P. Platonov on the occasion of his 75th Birthday}

\section{Introduction}
\label{sec:introduction}

We prove the main Conjecture 4 of our paper \cite{Nik9}.
See Theorem \ref{maintheorem} below. Further, we apply results of
\cite{Nik9}, and Theorem \ref{maintheorem} and its proof to classification of
degenerations of codimension one of K\"ahlerian K3 surfaces with finite
symplectic automorphism groups. By classification, we understand an enumeration of
connected components of the corresponding moduli.
See Theorems \ref{th:transuniqe}, \ref{th:moduliweakST}, \ref{th:modulistrongST} below.

Our papers \cite{Nik-1} --- \cite{Nik9},
ideas by K. Hashimoto in \cite{Hash},
results by R. Miranda, D.R. Morrison \cite{MM1}, \cite{MM2},
D.G. James \cite{J} and other results are important to us.


\section{Types of degenerations of codimension 1 of\\
K\"ahlerian K3 surfaces with finite symplectic\\
automorphism groups, and the main classification
Theorem \ref{maintheorem} which was conjectured in\\
\cite[Conjecture 4]{Nik9}}
\label{sec:conj}

Let $X$ be a K\"ahlerian K3 surface (e. g. see \cite{Sh}, \cite{PS},
\cite{BR}, \cite{Siu}, \cite{Tod}
about such surfaces). That is $X$ is a non-singular compact
complex surface with the trivial canonical class $K_X$, and
its irregularity $q(X)$ is equal to 0. Then $H^2(X,\bz)$ with the intersection pairing
is an even unimodular lattice $L_{K3}$
of the signature $(3,19)$. The primitive sublattice
$S_X=H^2(X,\bz)\cap H^{1,1}(X)\subset H^2(X,\bz)$
is the {\it Picard lattice} of $X$ generated by first Chern
classes of all line bundles over $X$.

Let $G$ be a finite symplectic automorphism group of $X$. Here symplectic means
that for any $g\in G$, for a non-zero holomoprhic $2$-form $\omega_X\in
H^{2,0}(X)=\Omega^2[X]=\bc\omega_X$, one has $g^\ast(\omega_X)=\omega_X$.
For an $G$-invariant sublattice $M\subset H^2(X,\bz)$, we denote by
$M^G=\{x\in K\ |\ G(x)=x\}$ the {\it invariant sublattice of $M$,} and by
$M_G=(M^G)^\perp_M$ the {\it coinvariant sublattice of $M$.}
By \cite{Nik-1/2}, the coinvariant lattice $S_G=H^2(X,\bz)_G=(S_X)_G$ is
{\it Leech type lattice:} i. e.  it is negative definite,
it has no elements with square $(-2)$, $G$ acts trivially
on the discriminant group $A_{S_G}$, and $(S_G)^G=\{0\}$.
For a general pair $(X,G)$, the $S_G=S_X$, and non-general $(X,G)$
can be considered as K\"ahlerian K3 surfaces with the condition $S_G\subset S_X$
on the Picard lattice (in terminology of \cite{Nik-1/2}). The dimension of their
moduli is equal to $20-\rk S_G$.

Let $E\subset X$ be a non-singular irreducible rational curve (that is $E\cong \bp^1$).
It is equivalent to: $\alpha=cl(E)\in S_X$, $\alpha^2=-2$, $\alpha$ is effective
and $\alpha$ is numerically effective: $\alpha\cdot D\ge 0$ for every irreducible curve
$D$ on $X$ such that $cl(D)\not=\alpha$.

Let us consider the primitive sublattice $S=[S_G, G(\alpha)]_{pr}\subset S_X$
of $S_X$ generated by the coinvariant sublattice $S_G$ and all classes of the orbit $G(E)$.
We remind that primitive means that $S_X/S$ has not torsion.
Since $S_G$ has no elements with square $(-2)$, it follows that $\rk S=\rk S_G+1$
and $S=[S_G,\alpha]_{pr}\subset S_X$.

Let us {\it assume that the lattice $S=[S_G,\alpha]_{pr}$ is negative definite.} Then the
elements $G(\alpha)$ define the basis of the root system $\Delta(S)$ of
all elements with square $(-2)$ of $S$. All curves $G(E)$ of $X$
can be contracted to Du Val singularities of types of
connected components of the Dynkin diagram of the basis. The group $G$
will act on the corresponding singular K3 surface $\overline{X}$ with these Du Val
singularities. For a general such triplet $(X,G,G(E))$, the Picard lattice $S_X=S$,
and such triplets can be considered as {\it a degeneration of codimension $1$} of
K\"ahlerian K3 surfaces $(X,G)$ with the finite symplectic automorphism group $G$.
Really, the dimension of moduli of K\"ahlerian K3 surfaces with the condition $S\subset S_X$
on the Picard lattice is equal to $20-\rk S=20-\rk S_G-1$.

By Global Torelli Theorem for K3 surfaces \cite{PS}, \cite{BR}, the main invariants
of the degeneration is the {\it type of the abstract group $G$} which is equivalent
to the isomorphism class of the coninvariant lattice $S_G$,
and the type of the degeneration which is
equivalent to the Dynkin diagram of the basis $G(\alpha)$ or the Dynkin diagram
of the rational curves $G(E)$.

We can consider only the maximal finite symplectic automorphism group $G$ with the
same coinvariant lattice $S_G$, that is $G=Clos(G)$.
By Global Torelly Theorem for K3 surfaces, this
is equivalent to
$$
G|S_G=\{ g\in O(S_G)\ |\ g\ is\ identity\ on\ A_{S_G}=(S_G)^\ast/S_G \}.
$$
Indeed, $G$ and $Clos(G)$ have the same lattice $S_G$, the same orbits
$G(E)$ and $Clos(G)(E)$, and the same sublattice $S\subset S_X$.

In \cite{Nik9}, all types of $G=Clos(G)$ and types of degenerations (that is Dynkin
diagrams of the orbits $G(E)$) are described. They are described in Table 1 below
where $n$ gives types of possible $G=Clos(G)$, and we show all possible
types of degenerations at the corresponding rows.


Theorem 3 in \cite{Nik9} shows that for
a fixed type (defined by {\bf n})
of abstract finite symplectic group of automorphisms of K\"ahlerian K3 surfaces
and for a fixed type of degeneration
of codimension $1$ (type of Dynkin diagram), the
discriminant group $A_S=S^\ast/S$
of the corresponding lattice $S$ is always the same.
It is natural to suppose
that the more strong statement is valid: that the isomorphism class of
the lattice $S$ is defined uniquely.
But, exact calculations and considerations show that it is valid with
some few exceptions only (which are given in the Theorem \ref{maintheorem} below.

\begin{theorem}
%
Let $X$ be a K\"ahlerian
K3 surface with $S_X<0$. Let $G=Aut^0 X$ be the group of symplectic
automorphisms of $X$ (it is always finite) and $P(X)$ be a the set
of classes of non-singular rational curves on $X$. Let
$S_G=((S_X)^G)^\perp_{S_X}\subset S_X$ be the coinvariant sublattice. Assume that
$S_X$ is generated by $S_G$ and $P(X)$ up to finite index, and
$\rk S_X=\rk S_G+1$. Then the isomorphism class of the lattice $S=S_X$
is defined uniquely by the type of $G$ as an abstract group (equivalent to {\bf n})
and the type of the Dynkin diagram of $P(X)$. All possible types of $G$ (equivalent to
the invariant {\bf n}) and Dynkin
diagrams $P(X)$ are given in Table 1 below.

But, it is valid with the following and only the following two exceptions:

(1) {\bf n=34} (equivalently, $G\cong {\mathfrak S}_4$) and the degeneration of
the type $6\aaa_1$;

(2) {\bf n=10} (equivalently, $H\cong {D_8}$) and the degeneration of the type $2\aaa_1$.

In both these cases there are exactly two isomorphism classes of the
lattices $S$. They (and their genuses) are given in Table 1 below as degenerations
$(6\aaa_1)_I$, $(6\aaa_1)_{II}$ for {\bf n=34}, and $(2\aaa_1)_I$, $(2\aaa_1)_{II}$
for {\bf n=10}.
\label{maintheorem}
\end{theorem}

This result was conjectured in \cite[Conjecture 4]{Nik9}.
Below, we give a proof.

%

\section{Discriminant forms technique according to \cite{Nik1}}
\label{sec:discformtech}
We use notations, definitions and results of
\cite{Nik1} about lattices
(that is non-degenerate integral symmetric bilinear
forms). Below, we remind the main definitions and results of
\cite{Nik1} which we shall use in this paper.

Let $S$ be a lattice, that is a free $\bz$-module of a finite rank
with non-degenerate symmetric bilinear form $x\cdot y\in \bz$ for $x,y\in S$.
We denote $x^2=x\cdot x$ for $x\in S$. A lattice $S$ is even if $x^2$ is
even for any $x\in S$.  For $0\not= \lambda \in \bq$, we denote by $S(\lambda)$
the lattice with symmetric bilinear form $\lambda\, x\cdot y$ for $x,y\in S$
if it is integral.

By $\oplus$, we denote the orthogonal sum of lattices,
quadratic forms. For $k\in \bz$ and $k\ge 0$,
we denote by $kS$ the orthogonal sum
of $k$ copies of a lattice $S$ (in \cite{Nik1}, we denoted
the same lattice as $S^k$). We use similar notations for finite
quadratic, symmetric bilinear forms.

For a prime $p$, we denote by $\bz_p$ the ring of $p$-adic integers,
and by $\bq_p$ the field of $p$-adic numbers.

Let $S$ be a lattice. Then we have the canonical embedding
$S\subset S^\ast=Hom(S,\bz)$. It defines the (finite)
discriminant group $A_S=S^\ast/S$ of $S$. By continuing
the symmetric bilinear form of $S$ to $S^\ast$, we obtain
the non-degenerate finite symmetric bilinear form $b_S$ on $A_S$ with values
in $\bq/\bz$ and the finite quadratic form $q_S$ on $A_S$
with values in $\bq/2\bz$ if $S$ is even. They are called
{\it discriminant forms of $S$.}

We denote by $l(A)$ the minimal
number
of generators of a finite Abelian group $A$, and by $|A|$
its order.
For a prime $p$, we denote by ${q_{S}}_p=q_{S\otimes \bz_p}$ the
$p$-component of $q_S$ (equivalently,
the discriminant quadratic form of the $p$-adic lattice
$S\otimes \bz_p$).

A $p$-adic lattice $K({q_{S}}_p)$
or the rank $l({A_S}_p)$ with the discriminant quadratic form
${q_{S}}_p$ is denoted by $K({q_{S}}_p)$. It is unique, up to
isomorphisms,
for $p\not=2$, and for $p=2$, if
${q_{S}}_2\not\cong q_\theta^{(2)}(2)\oplus
q^\prime$ where $q_\theta^{(2)}(2)$ denotes a finite quadratic form on
a group of order $2$ (see notations below).
For $p=2$ and ${q_{S}}_2\cong q_\theta^{(2)}(2)\oplus
q^\prime$, there are exactly two lattices $K({q_{S}}_2)$; their
determinants are different by $5\mod (\bz_2^\ast)^2$.
See \cite[Theorem 1.9.1]{Nik1}.

By $<A>$ we denote a lattice defined by a matrix $A$. In particular,
$K_\theta^{(p)}(p^k)=\left< \theta p^k\right> $, $\theta\in \bz_p^\ast$,
and
$U^{(2)}(2^k)=\left<
\begin{array}{cc}
0 & 2^k\\
2^k & 0
\end{array}
\right>
$,\
$
V^{(2)}(2^k)=\left<
\begin{array}{cc}
2^{k+1} & 2^k\\
2^k & 2^{k+1}
\end{array}
\right>
$
are standard $p$-adic lattices of the rank $1$ and $2$.
By Jordan decomposition (e.g. see \cite{CS}), any $p$-adic lattice is
their orthogonal sum. By $q_\theta^{(p)}(p^k)$ and $b_\theta^{(p)}(p^k)$
we denote discriminant quadratic and bilinear forms of $K_\theta^{(p)}(p^k)$ respectively.
By $u_+(2^k)$, $u_-(2^k)$ and $v_+(2^k)$, $v_-(2^k)$ we denote discriminant quadratic,
bilinear forms of $U^{(2)}(2^k)$ and $V^{(2)}(2^k)$ respectively.

One has relations between these forms from \cite[Prop. 1.8.2]{Nik1}:
\begin{equation}
Relations\ (a) - (k):
\label{relations:a-k}
\end{equation}

(a) $2K_\theta^{(p)}(p^k)\cong 2K_{\theta^\prime}^{(p)}(p^k)$ if $p\not=2$;

(b) $2U^{(2)}(2^k)\cong 2V^{(2)}(2^k)$;

(c) $K_\theta^{(2)}(2^k)\oplus K_{\theta^\prime}^{(2)}(2^k)\cong
K_{5\theta}^{(2)}(2^k)\oplus K_{5\theta^\prime}^{(2)}(2^k)$;

(d) $2K_\theta^{(2)}(2^k)\oplus K_{\theta^\prime}^{(2)}(2^k)\cong
\left\{
\begin{array}{ll}
V^{(2)}(2^k)\oplus K_{-5\theta^\prime}^{(2)}(2^k) & if\ \theta^\prime \equiv \theta (\mod 4),\\
U^{(2)}(2^k)\oplus K_{-\theta^\prime}^{(2)}(2^k) & if\ \theta^\prime \equiv -\theta (\mod 4)
\end{array}\right.
$;

(e) $V^{(2)}(2^k)\oplus K_\theta^{(2)}(2^{k+1})\cong U^{(2)}(2^k)\oplus K_{5\theta}^{(2)}(2^{k+1})$;

(f) $K_\theta^{(2)}(2^k)\oplus V^{(2)}(2^{k+1})\cong K_{5\theta}^{(2)}(2^k)\oplus U^{(2)}(2^{k+1})$;

(g) $K_\theta^{(2)}(2^k)\oplus K_{\theta^\prime}^{(2)}(2^{k+1})\cong
K_{\theta+2\theta^\prime}^{(2)}(2^k)\oplus K_{5(\theta^\prime-2\theta)}^{(2)}(2^{k+1})$;

(h) $K_\theta^{(2)}(2^k)\oplus K_{\theta^\prime}^{(2)}(2^{k+2})\cong
K_{5\theta}^{(2)}(2^k)\oplus K_{5\theta^\prime}^{(2)}(2^{k+2})$;

(i) $q_\theta^{(2)}(2)\cong  q_{5\theta}^{(2)}(2)$;

(j) $b_\theta^{(2)}(2)\cong b_{\theta^\prime}^{(2)}(2)$,\ $u_-^{(2)}(2)\cong v_-^{(2)}(2)$,
$b_\theta^{(2)}(4)\cong b_{5\theta}^{(2)}(4)$;

(k) relations between finite quadratic and bilinear forms which follow from (a) --- (h)
if one considers discriminant forms.

By \cite{Nik1}, the numbers $(t_{(+)}, t_{(-)})$ of positive, negative
squares of $S\otimes \br$ and the
discriminant quadratic form $q_S$ define the {\it genus}
of an even lattice $S$, that is isomorphism classes of
$S\otimes \br$ and $S\otimes \bz_p$ for all prime $p$.

If $S$ is an even lattice, the signature  $sign\ S=t_{(+)}-t_{(-)}$ of $S$ modulo $8$ that is $t_{(+)}-t_{(-)}\mod 8\equiv sign\ q_S \mod 8$
is the invariant of $q_S$. In particular, $sign\ S\equiv 0\mod 8$ if $S$
is unimodular and even. We have (e. g. see \cite[Prop. 1.11.2]{Nik1})

\begin{equation}
sign\ q_\theta^{(p)}(p^k)\equiv k^2 (1-p)+4k\eta \mod 8,\ where\ p\not=2,\ and\
(-1)^\eta=\left(\frac{\ \theta\ }{\ p\ }\right);
\label{sign8,1}
\end{equation}
\begin{equation}
sign\ q_\theta^{(2)}(2^k)\equiv \theta+4\omega(\theta)k\mod\ 8,\ where\
\omega(\theta)\equiv\frac{\theta^2-1}{8}\mod 2;
\label{sign8,2}
\end{equation}
\begin{equation}
sign\ u_+^{(2)}(2^k)\equiv 0\mod 8,\ \
sign\ v_+^{(2)}(2^k)\equiv 4k\mod 8.
\label{sign8,3}
\end{equation}
In particular,
\begin{equation}
sign\ q_\theta^{(2)}(4)\equiv \theta\mod 8,\ \
sign\ u_+^{(2)}(4)\equiv sign\ v_+^{(2)}(4)\equiv 0\mod 8.
\label{sign8,4}
\end{equation}

To find the genus of a lattice $S$, we consider the Jordan decomposition
\begin{equation}
S^{(p)}=S\otimes \bz_p=\bigoplus_{k\ge 0}{S^{(p)}_k(p^k)}
\label{Jordanlat}
\end{equation}
where $S^{(p)}_k$ are unimodular $p$-adic lattices. Then the
$p$-component ${q_S}_p$ is orthogonal sum of discriminant
quadratic forms of $S^{(p)}_k(p^k)$
for $k\ge 1$ and non-zero $S^{(p)}_k$.

For $p\not=2$, by relations (a) --- (k) above,
\begin{equation}
S^{(p)}_k=(m_k-1)K_1^{(p)}(1)\oplus K_{\theta_k}^{(p)}(1)
\label{inv,p}
\end{equation}
where $m_k=\rk S^{(p)}_k$ (or $m_k$ is the size of $S^{(p)}_k$),
and $\theta_k=\det(S^{(p)}_k)\in \bz_p^\ast/(\bz_p^\ast)^2$.
Thus, $S^{(p)}_k$ is defined by invariants $m_k=\rk S^{(p)}_k$ and
the Kronecker symbol $\left(\frac{\theta_k}{p}\right)$
where $\theta_k=\det S^{(p)}_k$.
Then $q_{S^{(p)}_k(p^k)}=(m_k-1)q_1^{(p)}(p^k)\oplus q_{\theta_k}^{(p)}(p^k)$.
Like in \cite{CS}, we shortly denote this form by the symbol $(p^k)^{{\pm}m_k}$
where $\pm 1=\left(\frac{\theta_k}{p}\right)$.

For $p=2$, the lattice $S^{(2)}_k$ has the type II if it is orthogonal sum of
lattices $U^{(2)}(1)$ and $V^{(2)}(1)$, and it has the type I otherwise.
Up to isomorphisms, $S^{(2)}_k$ is defined by the type, size=$\rk S^{(2)}_k$,
$\det=\det\ S^{(2)}_k\in \bz_2^\ast/(\pm )(\bz_2^\ast)^2=\{1,5\}\mod (\pm)(\bz_2^\ast)^2$,
and $sign\mod 8\equiv sign\ q_{S^{(2)}_k(4)}\mod 8$. See (\ref{sign8,4}) about
calculations of $sign\mod 8$. Shortly, like in \cite{CS}, $S^{(2)}_k(2^k)$
is denoted by $(2^k)_{II}^{(-1)^\delta\cdot s}$ for $S^{(2)}_k$
of the type $II$, size=$s$,
$\det=5^\delta \mod (\pm)(\bz_2^\ast)^2$, and by $(2^k)_{sign\mod8}^{(-1)^\delta\cdot s}$
for $S^{(2)}_k$ of
type $I$, size=$s$, $\det=5^\delta \mod (pm)(\bz_2^\ast)^2$ and with $sign\mod 8$.
For the Jordan decomposition  (\ref{Jordanlat}), these components are separated by comme\ ,\
instead of $\oplus$. The same notations are used for the corresponding discriminant quadratic forms.

In Sect. \ref{sec:Appendix}, we give Program 8 which uses these invariants to calculate
the genus of a lattice given by an integral symmetric matrix $l$.

The following fact from \cite{Nik1}  will be very important for us.

\begin{proposition} (Proposition 1.6.1 from \cite{Nik1})
Primitive embeddings of an even lattice $M$  into
unimodular even lattices such that the orthogonal complement to $M$ is isomorphic to $K$
are defined by isomorphisms $\gamma:A_M\cong A_K$ such that $q_K\circ\gamma=-q_M$.
Two such isomorphims $\gamma$, $\gamma^\prime$ define isomorphic primitive embeddings if and
only if they are conjugate by an automorphism of $K$, and they define isomorphic primitive
sublattices if and only if $\gamma\circ \overline{\phi}=\overline{\psi}\circ \gamma^\prime$ for
some $\phi\in O(M)$, $\psi\in O(K)$.
\label{prop:primembunim}
\end{proposition}

\section{A proof of Theorem \ref{maintheorem}}
\label{proofmaintheorem}.

We remind that Niemeier lattices are negative definite even unimodular lattices
of the rank $24$. There are $24$ such lattices, up to isomorphisms,
$N=N_j$, $j=1,2,\dots,24$, classified by Niemeier. They are characterized
by their root sublattices $N^{(2)}$ generated by all their elements
with square $(-2)$ (called roots). Further, $\Delta(N)$ is the set of all roots of $N$.
We have the following list of Niemeier lattices $N_j$ where the number $j$ is shown in the bracket:
$$
N^{(2)}=[\Delta(N)]=
$$
$$
(1)\ D_{24},\ (2)\ D_{16}\oplus E_8,\ (3)\ 3E_8,\ (4)\ A_{24},\
(5)\ 2D_{12},\ (6)\ A_{17}\oplus E_7, \ (7)\ D_{10}\oplus 2E_7,
$$
$$
(8)\ A_{15}\oplus D_9,\
(9)\ 3D_8,\
(10)\ 2A_{12},\ (11)\ A_{11}\oplus D_7\oplus E_6,\ (12)\ 4E_6,\
(13)\ 2A_9\oplus D_6,\
$$
$$
(14)\ 4D_6,\
(15)\ 3A_8,\ (16)\ 2A_7\oplus 2D_5,\ (17)\ 4A_6,\ (18)\ 4A_5\oplus D_4,\
(19)\ 6D_4,
$$
$$
(20)\ 6A_4,\
(21)\ 8A_3,\ (22)\ 12A_2,\ (23)\ 24A_1
$$
give $23$ Niemeier lattices $N_j$. The last is Leech lattice (24)
with $N^{(2)}=\{0\}$ which has no roots. Further,
$N(R)$ denotes the Niemeier lattice with the root system $R$. We
fix the basis $P(N)$ of the root system $\Delta(N)$ of $N$. By
$A(N)\subset O(N)$ we denote the subgroup of the group $O(N)$ of
automorphisms of $N$ which permute the basis $P(N)$.

\medskip

Let us consider the lattice $S$ for one of types (given by {\bf n}) of
the group $G$ and for a type of the degeneration given by a Dynkin diagram $R$.
All such possibilities are enumerated in Table 1.

We consider all possible markings $S\subset N_j$ by Niemeier lattices.
It means that: $S\subset N_j$ is a primitive sublattice that is $N_j/S$ has
no torsion; $P(S)\subset P(N_j)$.

By \cite[Remark 1.14.7]{Nik1},
$$
G=\{g\in A(N_j)\ |\ g|S^\perp_{N_j}=identity\}\subset A(N_j).
$$

It follows that we can find all such lattices $S$ as follows.

Find a Niemeier lattice $N_j$, a subgroup $G\subset A(N_j)$
(up to conjugacy in $A(N_j))$, an element $\alpha \in P(N_j)$
such that $G(\alpha)$ has the Dynkin diagram $R$, and the
primitive sublattice $S=[(N_j)_G,\alpha]_{pr}\subset N_j$
has a primitive embedding into $L_{K3}$. Then $S$, $G|S$ and
$G(\alpha)\subset P(S)$ correspond to a degeneration of
codimenion one of some K\"ahlerian K3
surfaces by Global Torelli Theorem and epimorphicity of Torelli
map for K\"ahlerian K3 surfaces \cite{BR}, \cite{Kul}, \cite{PS}, \cite{Siu}, \cite{Tod}.

Obviously, the isomorphism class of $S$ does not change if $G$
is changed by conjugacy in $A(N_j)$, the element $\alpha$ is
changed to $h(\alpha)$ by $h\in Normalizer(A(N_j),G)$.

All such triplets $(N_j,G,\alpha)$ (up to isomorphisms)
are shown in columns of Table 2 below using results of \cite{Nik7}, \cite{Nik8}
and \cite{Nik9} and the program GAP, \cite{GAP}. In Table 2, for all possible {\bf n}
and the type $R$ (Dynkin diagram) of the degeneration, the first line of the
column  gives $j$, the second line gives the group $G=H_{n,t}\subset A(N_j)$
(in notations of \cite{Nik7}, \cite{Nik8}
and \cite{Nik9}). The third line gives $\alpha\in P(N_j)$ which give
different orbits $G(\alpha)\subset P(N_j)$ with the Dynkin diagram
of the type $R$, but these orbits are conjugate by the $Normalizer(A(N_j),G=H_{n,t})$.
Thus, $S=[(N_j)_{H_{n,t}},\alpha]_{pr}\subset N_j$, up to isomorphisms.
The fourth line gives the Dynkin diagram of the root system $(S^\perp_{N_j})^{(2)}$ of
elements with square $(-2)$ in $S^\perp_{N_j}$. The fifth line
(if it is necessary) gives the number of elements with square $(-4)$
in $S^\perp_{N_j}$.

In Table 1 below we calculate the genus of the lattices $S$ for all possible triplets
$(N_j,G,\alpha)$ (equivalently, columns of Table 2),
using invariants and relations (\ref{relations:a-k}) of Sect. \ref{sec:discformtech}.
We use Programs 7 and 8 from Appendix, Sect. \ref{sec:Appendix}. The genus is defined by the types
{\bf n} and $R$ of the degeneration except two cases. For {\bf n}$=10$, $R=2\aaa_1$ and
{\bf n}$=34$, $R=6\aaa_1$ there are
two possibilities for the genus which we label by $(2\aaa_1)_I$, $(2\aaa_1)_{II}$ and
by $(6\aaa_1)_I$, $(6\aaa_1)_{II}$ respectively. Here $I$ and $II$ show the type of the
corresponding $2$-adic lattices. In Table 1, we also give the genus of the lattice $S_G$
which was calculated by K. Hashimoto in \cite{Hash} (it is useful to compare genuses of
$S_G$ and $S$).

Let us consider a case $({\bf n},R)$ which is different from $(4,\aaa_1)$ and $(16,\aaa_1)$.
In Table 2, one of columns of this case $({\bf n},R)$ is marked by $\ast$ from above. We
denote the lattice $S$ of this case by ${\bf S}$.  The
orthogonal complement ${\bf S}_{N_j}^\perp$ of this case either has the root system $({\bf S}_{N_j}^\perp)^{(2)}$ of
elements with square $(-2)$ which is different from all other columns of this case, or
it has different number of elements with square $(-4)$ (the last happens for $(10,(2\aaa_1)_{II})$,
$(34,(6\aaa_1)_{II})$ and $(51,8\aaa_1)$).  Since genuses of lattices $S$ of all other columns and
the genus of ${\bf S}$ are the same, by Proposition \ref{prop:primembunim}, there exists an isomorphism $\gamma:A_S\cong A_{{\bf S}_{N_j}^\perp}$
such that $q_S\circ \gamma=-q_{{\bf S}_{N_j}^\perp}$. By Proposition
\ref{prop:primembunim}, this defines a primitive embedding
$S\subset N$ into one of Niemeier lattices $N$ such that $S_N^\perp={\bf S}_{N_j}^\perp$. Changing
this embedding by $w\in W(N)$ for the reflection group $W(N)$ of $N$, if necessary,
we can assume that $P(S)\subset P(N)$ and $S\subset N$ is isomorphic to one of columns of
the case $({\bf n},R)$. Since $S^\perp_N\cong {\bf S}_{N_j}^\perp$ and the column with this
property is unique, we obtain that $S\cong {\bf S}$. Thus, the lattice $S$ is unique up
to isomorphisms.

For the cases $({\bf n}=4,\aaa_1)$ and $({\bf n}=16,\aaa_1)$
(and all other cases $({\bf n}, \aaa_1)$ as well) we have that
$K=S_G\oplus \left< -2\right>\subset S$ is an overlattice
of a finite index, by definition.
Discriminant groups of $S_G$ and their orders are is known (e.g.
see  \cite{Hash}) (they can be found from the Table 1).
Discriminant groups of $S$ and their orders are calculated in \cite{Nik9}
(they can be found from the Table 1). It follows that orders
of the discriminant groups of the lattices $K$ and $S$ are the same.
It follows that $S=K=S_G\oplus \left<-2\right>$, and it is
unique up to isomorphisms since the lattices $S_G$ are unique up to isomorphisms
by \cite{Hash}.

This finishes the proof of Theorem \ref{maintheorem}.


\begin{table}
\label{table1}
\caption{Genuses of degenerations of codimension $1$ of K\"ahlerian K3 surfaces
with finite symplectic automorphism groups $G=Clos(G)$.}



\begin{tabular}{|c||c|c|c|c|c|c|c|c|}
\hline
 {\bf n}& $|G|$& $i$&  $G$  & $\rk S_G$ &$q_{S_G}$       &$Deg$& $\rk S$ &$q_S$            \\
\hline
\hline
 $1$  &$2$ & $1$& $C_2$ &   $8$     &$2_{II}^{+8}$ & $\aaa_1$   &  $9$  &
$2_7^{+9}$\\
\hline
      &    &    &    &              &              & $2\aaa_1$  &  $9$  &
$2_{II}^{-6},4_3^{-1}$\\
\hline
\hline
 $2$  & $3$& $1$&$C_3$  &    $12$   & $3^{+6}$       & $\aaa_1$   &  $13$ &
$2_3^{-1},3^{+6}$\\
\hline
      &    &    &       &           &                & $3\aaa_1$  & $13$  &
$2_1^{+1},3^{-5}$\\
\hline
\hline
 $3$  &$4$ & $2$&$C_2^2$&    $12$   & $2_{II}^{-6},4_{II}^{-2}$& $\aaa_1$   & $13$  &
$2_3^{+7},4_{II}^{+2}$\\
\hline
      &    &    &       &           &                            & $2\aaa_1$& $13$  &
$2_{II}^{-4},4_7^{-3}$\\
\hline
      &    &    &       &           &                & $4\aaa_1$  & $13$  &
$2_{II}^{-6},8_3^{-1}$\\
\hline
\end{tabular}
\end{table}

\begin{table}

\begin{tabular}{|c||c|c|c|c|c|c|c|c|}
\hline
 {\bf n}& $|G|$& $i$&  $G$   & $\rk S_G$ &$q_{S_G}$       &$Deg$& $\rk S$ &$q_S$            \\
\hline
\hline
 $17$&$12$ & $3$&$\AAA_4$&  $16$     & $2_{II}^{-2},4_{II}^{-2},3^{+2}$ & $\aaa_1$ &  $17$  &
$2_7^{-3},4_{II}^{+2},3^{+2}$\\
\hline
     &     &    &        &           &                                  &$3\aaa_1$ &  $17$  &
$2_1^{-3},4_{II}^{+2},3^{-1}$\\
\hline
     &     &    &        &           &                                  &$4\aaa_1$ &  $17$  &
$2_{II}^{-2},8_3^{-1},3^{+2}$\\
\hline
     &     &    &        &           &                                  &$6\aaa_1$ &  $17$  &
$4_1^{-3},3^{+1}$\\
\hline
     &     &    &        &           &                                  &$12\aaa_1$&  $17$  &
$2_{II}^{-2},8_1^{+1},3^{-1}$\\
\hline
\hline
 $18$&$12$ &$4$ &$D_{12}$&  $16$     &$2_{II}^{+4},3^{+4}$ &  $\aaa_1$     &  $17$  &
$2_7^{+5},3^{+4}$\\
\hline
     &     &    &        &           &                     &  $2\aaa_1$    &  $17$  &
$2_{II}^{+2},4_7^{+1},3^{+4}$\\
\hline

\end{tabular}

\end{table}

\begin{table}

\begin{tabular}{|c||c|c|c|c|c|c|c|c|}
\hline
 {\bf n}& $|G|$& $i$&  $G$   & $\rk S_G$ &$q_{S_G}$       &$Deg$& $\rk S$ &$q_S$            \\
\hline
\hline
 $39$&$32$ &$27$&$2^4C_2$&  $17$     &$2_{II}^{+2},4_0^{+2},8_7^{+1}$ & $4\aaa_1$&$18$&
$4_6^{+4}$\\
\hline
     &     &    &        &           &                                & $8\aaa_1$&  $18$&
$2_{II}^{+2},4_7^{+1},8_7^{+1}$\\
\hline
     &     &    &        &           &                                & $16\aaa_1$&$18$&
$2_{II}^{+2},4_6^{+2}$\\
\hline
\hline
$40$ &$32$ &$49$&$Q_8*Q_8$& $17$     & $4_7^{+5}$ & $8\aaa_1$  & $18$ &
$4_6^{+4}$\\
\hline
\hline
$46$ &$36$ &$9$ &$3^2C_4$ & $18$     & $2_6^{-2},3^{+2},9^{-1}$& $6\aaa_1$   &  $19$  &
$4_7^{+1},3^{+1},9^{-1}$               \\
\hline
     &     &    &        &           &                           & $9\aaa_1$ &  $19$  &
$2_5^{-3},3^{+2}$ \\
\hline
     &     &    &        &           &                           & $9\aaa_2$ &  $19$  &
$2_5^{-1},3^{+2}$\\
\hline

\end{tabular}

\end{table}

\newpage

\begin{table}
\label{table2}
\caption{Lattices $S^\perp_{N_j}$ for markings by Niemeier lattices
$N_j$.}
\end{table}

{\bf n}=1, degeneration $\aaa_1$:

\begin{center}
\begin{tabular}{|c|c|c|c|c|c|c|c|c|c|c|c|c|c|}
\hline
$j$&$3$&$3$&$3$&$3$&$3$&$3$&$3$&$3$&$6$&$6$            \\
\hline
$H$&$H_{1,1}$&$H_{1,1}$&$H_{1,1}$&$H_{1,1}$&$H_{1,1}$&$H_{1,1}$&$H_{1,1}$&$H_{1,1}$&
$H_{1,1}$&$H_{1,1}$           \\
\hline
orbit of&$\alpha_{1,3}$&$\alpha_{2,3}$&$\alpha_{3,3}$&$\alpha_{4,3}$&$\alpha_{5,3}$&
$\alpha_{6,3}$&$\alpha_{7,3}$&$\alpha_{8,3}$&
$\alpha_{1,2}$&$\alpha_{2,2}$             \\
\hline
$(S^\perp_{N_j})^{(2)}$&$E_7$&$E_7$&$E_7$&$E_7$&$E_7$&$E_7$&$E_7$&$E_7$&
$9A_1\oplus D_6$&$9A_1\oplus D_6$ \\
\hline
\end{tabular}
\end{center}

\begin{center}
\begin{tabular}{|c|c|c|c|c|c|c|c|c|c|}
\hline
$6$&$6$&$6$&$6$&$6$&$6$&$7$            \\
\hline
$H_{1,1}$&$H_{1,1}$&$H_{1,1}$&$H_{1,1}$&$H_{1,1}$&$H_{1,1}$&$H_{1,1}$                \\
\hline
$\alpha_{3,2}$&$\alpha_{4,2}$&$\alpha_{5,2}$&$\alpha_{6,2}$&$\alpha_{7,2}$&
$\alpha_{9,1}$&$\alpha_{1,1}$\\
\hline
$9A_1\oplus D_6$&$9A_1\oplus D_6$&$9A_1\oplus D_6$&$9A_1\oplus D_6$&$9A_1\oplus D_6$&
$8A_1\oplus E_7$&$D_7\oplus A_1$        \\
\hline
\end{tabular}
\end{center}

\begin{center}
\begin{tabular}{|c|c|c|c|c|c|c|c|c|c|c|c|c|c|}
\hline
$7$&$7$&$7$&$7$&$7$&$7$&$7$&$8$            \\
\hline
$H_{1,1}$&$H_{1,1}$&$H_{1,1}$&$H_{1,1}$&$H_{1,1}$&$H_{1,1}$&$H_{1,1}$&
$H_{1,1}$              \\
\hline
$\alpha_{2,1}$&$\alpha_{3,1}$&$\alpha_{4,1}$&$\alpha_{5,1}$&
$\alpha_{6,1}$&$\alpha_{7,1}$&$\alpha_{8,1}$&$\alpha_{1,2}$             \\
\hline
$D_7\oplus A_1$&$D_7\oplus A_1$&$D_7\oplus A_1$&$D_7\oplus A_1$&
$D_7\oplus A_1$&$D_7\oplus A_1$&$D_7\oplus A_1$&$9A_1\oplus D_6$                 \\
\hline
\end{tabular}
\end{center}

\begin{center}
\begin{tabular}{|c|c|c|c|c|c|c|c|c|c|c|}
\hline
$8$&$8$&$8$&$8$&$8$&$8$&$8$ \\
\hline
$H_{1,1}$&$H_{1,1}$&$H_{1,1}$&$H_{1,1}$&$H_{1,1}$&$H_{1,1}$&$H_{1,1}$                \\
\hline
$\alpha_{2,2}$&$\alpha_{3,2}$&$\alpha_{4,2}$&$\alpha_{5,2}$
&$\alpha_{6,2}$&$\alpha_{7,2}$&$\alpha_{8,1}$ \\
\hline
$9A_1\oplus D_6$&$9A_1\oplus D_6$&$9A_1\oplus D_6$&$9A_1\oplus D_6$&$9A_1\oplus D_6$&
$9A_1\oplus D_6$&$7A_1\oplus D_8$        \\
\hline
\end{tabular}
\end{center}

\begin{center}
\begin{tabular}{|c|c|c|c|c|c|c|c|c|c|c|}
\hline
$9$&$9$&$9$&$9$&$9$&$9$&$9$&$9$                            \\
\hline
$H_{1,1}$&$H_{1,1}$&$H_{1,1}$&$H_{1,1}$&$H_{1,1}$&$H_{1,1}$&$H_{1,1}$&$H_{1,1}$  \\
\hline
$\alpha_{1,3}$&$\alpha_{2,3}$&$\alpha_{3,3}$&$\alpha_{4,3}$&$\alpha_{5,3}$&
$\alpha_{6,3}$&$\alpha_{7,3}$&$\alpha_{8,3}$                          \\
\hline
$A_1\oplus D_6$&$A_1\oplus D_6$&$A_1\oplus D_6$&$A_1\oplus D_6$&$A_1\oplus D_6$&$A_1\oplus D_6$&
$A_1\oplus D_6$&$A_1\oplus D_6$        \\
\hline
\end{tabular}
\end{center}

\begin{center}
\begin{tabular}{|c|c|c|c|c|c|c|c|c|c|c|c|}
\hline
$11$&$11$&$11$&$11$&$11$&$11$ \\
\hline
$H_{1,1}$&$H_{1,1}$&$H_{1,1}$&$H_{1,1}$&$H_{1,1}$&$H_{1,1}$       \\
\hline
$\alpha_{6,1}$&$\alpha_{1,2}$&$\alpha_{2,2}$&$\alpha_{3,2}$&$\alpha_{4,2}$&$\alpha_{5,2}$       \\
\hline
$5A_1\oplus D_4\oplus D_6$&$7A_1\oplus 2D_4$&$7A_1\oplus 2D_4$&$7A_1\oplus 2D_4$&$7A_1\oplus 2D_4$&
$7A_1\oplus 2D_4$        \\
\hline
\end{tabular}
\end{center}

\begin{center}
\begin{tabular}{|c|c|c|c|c|c|c|c|c|c|c|}
\hline
$11$&$11$&$12$&$12$&$12$&$12$                     \\
\hline
$H_{1,1}$&$H_{1,1}$&$H_{1,1}$&$H_{1,1}$&$H_{1,2}$&$H_{1,2}$            \\
\hline
$\alpha_{2,3}$&$\alpha_{4,3}$&$\alpha_{2,1}$,$\alpha_{2,2}$,$\alpha_{2,3}$,$\alpha_{2,4}$&
$\alpha_{4,1}$,$\alpha_{4,2}$,$\alpha_{4,3}$,$\alpha_{4,4}$&
$\alpha_{2,1}$&$\alpha_{4,1}$  \\
\hline
$9A_1\oplus D_6$&$9A_1\oplus D_6$&$3A_1\oplus 3D_4$&$3A_1\oplus D_4$&
$3A_1\oplus E_6$&$3A_1\oplus E_6$        \\
\hline
\end{tabular}
\end{center}

\begin{center}
\begin{tabular}{|c|c|c|c|c|c|c|c|c|c|c|}
\hline
$12$&$12$&$12$&$12$&$13$&$13$&$13$                \\
\hline
$H_{1,2}$&$H_{1,2}$&$H_{1,2}$&$H_{1,2}$&
$H_{1,1}$&$H_{1,1}$&$H_{1,1}$             \\
\hline
$\alpha_{1,2}$,$\alpha_{6,2}$&$\alpha_{2,2}$&$\alpha_{3,2}$,$\alpha_{5,2}$&
$\alpha_{4,2}$&
$\alpha_{5,1}$,$\alpha_{5,2}$&$\alpha_{1,3}$&
$\alpha_{2,3}$              \\
\hline
$A_5\oplus D_4$&$A_5\oplus D_4$&$A_5\oplus D_4$&$A_5\oplus D_4$&
$9A_1\oplus D_6$&$11A_1\oplus D_4$&$11A_1\oplus D_4$        \\
\hline
\end{tabular}
\end{center}

\begin{center}
\begin{tabular}{|c|c|c|c|c|c|c|c|c|c|c|}
\hline
$13$&$13$&$13$&$14$&$14$&$14$              \\
\hline
$H_{1,1}$&$H_{1,1}$&$H_{1,1}$&$H_{1,1}$&$H_{1,1}$&$H_{1,1}$         \\
\hline
$\alpha_{3,3}$&$\alpha_{4,3}$&$\alpha_{5,3}$,$\alpha_{6,3}$&
$\alpha_{1,1}$,$\alpha_{1,2}$&$\alpha_{2,1}$,$\alpha_{2,2}$&
$\alpha_{3,1}$,$\alpha_{3,2}$    \\
\hline
$11A_1\oplus D_4$&$11A_1\oplus D_4$&$11A_1\oplus D_4$&
$A_1\oplus A_3\oplus D_5$&$A_1\oplus A_3\oplus D_5$&$A_1\oplus A_3\oplus D_5$ \\
\hline
\end{tabular}
\end{center}

\begin{center}
\begin{tabular}{|c|c|c|c|c|c|c|c|c|c|c|}
\hline
$14$&$15$&$15$&$15$&$15$&$16$&$16$&$16$              \\
\hline
$H_{1,1}$&$H_{1,1}$&$H_{1,1}$&$H_{1,1}$&$H_{1,1}$&
$H_{1,1}$&$H_{1,1}$&$H_{1,1}$             \\
\hline
$\alpha_{4,1}$,$\alpha_{4,2}$&$\alpha_{1,1}$,$\alpha_{8,1}$&$\alpha_{2,1}$,$\alpha_{7,1}$&
$\alpha_{3,1}$,$\alpha_{6,1}$&$\alpha_{4,1}$,$\alpha_{5,1}$&
$\alpha_{1,1}$,$\alpha_{7,1}$&$\alpha_{2,1}$,$\alpha_{6,1}$&$\alpha_{3,1}$,$\alpha_{5,1}$           \\
\hline
$A_1\oplus A_3\oplus D_5$&$A_6$&$A_6$&$A_6$&$A_6$&
$4A_1\oplus A_5$&$4A_1\oplus A_5$&$4A_1\oplus A_5$        \\
\hline
\end{tabular}
\end{center}

\begin{center}
\begin{tabular}{|c|c|c|c|c|c|c|c|c|c|c|}
\hline
$16$&$16$&$16$&$16$&$16$&$16$&$16$          \\
\hline
$H_{1,1}$&$H_{1,1}$&$H_{1,2}$&$H_{1,2}$&$H_{1,2}$&$H_{1,2}$&
$H_{1,3}$                \\
\hline
$\alpha_{4,1}$&$\alpha_{4,2}$&
$\alpha_{4,1}$,$\alpha_{4,2}$&$\alpha_{1,3}$,$\alpha_{1,4}$&
$\alpha_{2,3}$,$\alpha_{2,4}$&$\alpha_{3,3}$,$\alpha_{3,4}$&
$\alpha_{1,3}$           \\
\hline
$4A_1\oplus A_5$&$3A_1\oplus A_7$&
$7A_1\oplus 2D_4$&$11A_1\oplus D_4$&$11A_1\oplus D_4$&$11A_1\oplus D_4$&
$A_1\oplus A_3\oplus D_4$         \\
\hline
\end{tabular}
\end{center}

\begin{center}
\begin{tabular}{|c|c|c|c|c|c|c|c|c|c|c|}
\hline
$16$&$16$&$16$&$16$&$16$&$16$           \\
\hline
$H_{1,3}$&$H_{1,3}$&$H_{1,3}$&$H_{1,3}$&$H_{1,3}$&$H_{1,3}$          \\
\hline
$\alpha_{2,3}$&$\alpha_{3,3}$&$\alpha_{4,3}$,$\alpha_{5,3}$&
$\alpha_{1,4}$&$\alpha_{2,4}$&$\alpha_{3,4}$             \\
\hline
$A_1\oplus A_3\oplus D_4$&$A_1\oplus A_3\oplus D_4$&$A_1\oplus A_3\oplus D_4$&
$3A_1\oplus D_5$&$3A_1\oplus D_5$&$3A_1\oplus D_5$        \\
\hline
\end{tabular}
\end{center}

\begin{center}
\begin{tabular}{|c|c|c|c|c|c|c|c|c|c|c|}
\hline
$18$&$18$&$18$&$18$&$18$&$18$                              \\
\hline
$H_{1,1}$&$H_{1,1}$&$H_{1,1}$&$H_{1,1}$&$H_{1,1}$&$H_{1,1}$              \\
\hline
$\alpha_{1,1}$,$\alpha_{5,1}$&$\alpha_{2,1}$,$\alpha_{4,1}$&$\alpha_{3,1}$&
$\alpha_{3,2}$&$\alpha_{2,5}$&$\alpha_{4,5}$             \\
\hline
$3A_1\oplus 2A_3$&$3A_1\oplus 2A_3$&$3A_1\oplus 2A_3$&
$2A_1\oplus A_3\oplus A_5$&$4A_1\oplus A_5$&$4A_1\oplus A_5$        \\
\hline
\end{tabular}
\end{center}

\begin{center}
\begin{tabular}{|c|c|c|c|c|c|c|c|c|c|c|}
\hline
$18$&$18$&$18$&
$19$&$19$            \\
\hline
$H_{1,2}$&$H_{1,2}$&$H_{1,2}$&
$H_{1,1}$&$H_{1,1}$                \\
\hline
$\alpha_{3,1}$,$\alpha_{3,2}$,$\alpha_{3,3}$,$\alpha_{3,4}$&
$\alpha_{1,5}$,$\alpha_{3,5}$,$\alpha_{4,5}$&$\alpha_{2,5}$&
$\alpha_{2,3}$,$\alpha_{2,4}$,$\alpha_{2,5}$,$\alpha_{2,6}$&
$\alpha_{4,3}$,$\alpha_{4,4}$,$\alpha_{4,5}$,$\alpha_{4,6}$            \\
\hline
$11A_1\oplus D_4$&$15A_1$&$15A_1$&
$A_1\oplus 3A_3$&$A_1\oplus 3A_3$       \\
\hline
\end{tabular}
\end{center}

\begin{center}
\begin{tabular}{|c|c|c|c|c|c|c|c|c|c|c|}
\hline
$19$&$19$&$20$&$20$            \\
\hline
$H_{1,2}$&$H_{1,2}$&$H_{1,1}$&$H_{1,1}$              \\
\hline
$\alpha_{1,3}$,$\alpha_{3,3}$,$\alpha_{4,3}$,$\alpha_{1,4}$,$\alpha_{3,4}$,$\alpha_{4,4}$&
$\alpha_{2,3}$,$\alpha_{2,4}$&
$\alpha_{1,1}$,$\alpha_{4,1}$,$\alpha_{1,2}$,$\alpha_{4,2}$&
$\alpha_{2,1}$,$\alpha_{3,1}$,$\alpha_{2,2}$,$\alpha_{3,2}$             \\
\hline
$3A_1\oplus D_4$&$3A_1\oplus D_4$&
$A_2\oplus A_4$&$A_2\oplus A_4$        \\
\hline
\end{tabular}
\end{center}

\begin{center}
\begin{tabular}{|c|c|c|c|c|c|c|c|c|c|c|}
\hline
$21$&$21$&$21$&$21$            \\
\hline
$H_{1,1}$&$H_{1,2}$&$H_{1,2}$&$H_{1,2}$                \\
\hline
$\alpha_{2,1}$,$\alpha_{2,2}$,$\alpha_{2,3}$,$\alpha_{2,4}$,$\alpha_{2,5}$,$\alpha_{2,6}$,$\alpha_{2,7}$,$\alpha_{2,8}$&
$\alpha_{1,1}$,$\alpha_{3,1}$,$\alpha_{1,5}$,$\alpha_{3,5}$&
$\alpha_{2,1}$,$\alpha_{2,5}$&$\alpha_{2,2}$,$\alpha_{2,4}$ \\
\hline
$15A_1$&$5A_1\oplus A_3$&$5A_1\oplus A_3$&$3A_1\oplus 2A_3$       \\
\hline
\end{tabular}
\end{center}

\begin{center}
\begin{tabular}{|c|c|c|c|c|c|c|c|c|c|c|}
\hline
$22$&$23$ $\ast$                              \\
\hline
$H_{1,1}$&$H_{1,1}$                             \\
\hline
$\alpha_{1,1}$,$\alpha_{2,1}$,$\alpha_{1,2}$,$\alpha_{2,2}$,$\alpha_{1,5}$,$\alpha_{2,5}$,$\alpha_{1,11}$,$\alpha_{2,11}$&
$\alpha_1$,$\alpha_4$,$\alpha_8$,$\alpha_{12}$,$\alpha_{15}$,$\alpha_{17}$,$\alpha_{18}$,$\alpha_{21}$   \\
\hline
$3A_2$&$7A_1$                 \\
\hline
\end{tabular}
\end{center}

\medskip

{\bf n}=1, degeneration $2\aaa_1$:

\begin{center}
\begin{tabular}{|c|c|c|c|c|c|c|c|c|c|c|c|}
\hline
$j$&$3$&$3$&$3$&$3$&$3$&$3$&$3$&$3$&$6$&$6$            \\
\hline
$H$&$H_{1,1}$&$H_{1,1}$&$H_{1,1}$&$H_{1,1}$&$H_{1,1}$&$H_{1,1}$&$H_{1,1}$&$H_{1,1}$&
$H_{1,1}$&$H_{1,1}$\\
\hline
orbit of&$\alpha_{1,1}$&$\alpha_{2,1}$&$\alpha_{3,1}$&$\alpha_{4,1}$&$\alpha_{5,1}$&
$\alpha_{6,1}$&$\alpha_{7,1}$&$\alpha_{8,1}$&$\alpha_{1,1}$&$\alpha_{2,1}$             \\
\hline
$(S^\perp_{N_j})^{(2)}$&$E_8$&$E_8$&$E_8$&$E_8$&$E_8$&$E_8$&$E_8$&$E_8$&
$7A_1\oplus E_7$&$7A_1\oplus E_7$        \\
\hline
\end{tabular}
\end{center}

\begin{center}
\begin{tabular}{|c|c|c|c|c|c|c|c|c|c|}
\hline
$6$&$6$&$6$&$6$&$6$&$6$&$7$&$7$            \\
\hline
$H_{1,1}$&$H_{1,1}$&$H_{1,1}$&$H_{1,1}$&$H_{1,1}$&$H_{1,1}$&
$H_{1,1}$&$H_{1,1}$                \\
\hline
$\alpha_{3,1}$&$\alpha_{4,1}$&$\alpha_{5,1}$&$\alpha_{6,1}$&$\alpha_{7,1}$&$\alpha_{8,1}$&
$\alpha_{9,1}$&$\alpha_{1,2}$             \\
\hline
$7A_1\oplus E_7$&$7A_1\oplus E_7$&$7A_1\oplus E_7$&$7A_1\oplus E_7$&$7A_1\oplus E_7$&
$7A_1\oplus E_7$&$D_8$&$D_9$                 \\
\hline
\end{tabular}
\end{center}

\begin{center}
\begin{tabular}{|c|c|c|c|c|c|c|c|c|c|c|c|}
\hline
$7$&$7$&$7$&$7$&$7$&$7$&$8$&$8$&$8$&$8$            \\
\hline
$H_{1,1}$&$H_{1,1}$&$H_{1,1}$&$H_{1,1}$&$H_{1,1}$&$H_{1,1}$&
$H_{1,1}$&$H_{1,1}$&$H_{1,1}$&$H_{1,1}$         \\
\hline
$\alpha_{2,2}$&$\alpha_{3,2}$&$\alpha_{4,2}$&$\alpha_{5,2}$&
$\alpha_{6,2}$&$\alpha_{7,2}$&$\alpha_{1,1}$&$\alpha_{2,1}$&$\alpha_{3,1}$&$\alpha_{4,1}$             \\
\hline
$D_9$&$D_9$&$D_9$&$D_9$&$D_9$&$D_9$&
$6A_1\oplus D_8$&$6A_1\oplus D_8$&$6A_1\oplus D_8$&$6A_1\oplus D_8$                 \\
\hline
\end{tabular}
\end{center}

\begin{center}
\begin{tabular}{|c|c|c|c|c|c|c|c|c|c|c|}
\hline
$8$&$8$&$8$&$8$&$9$&$9$&$9$&$9$&$9$&$9$  \\
\hline
$H_{1,1}$&$H_{1,1}$&$H_{1,1}$&$H_{1,1}$&$H_{1,1}$&$H_{1,1}$&$H_{1,1}$&$H_{1,1}$&
$H_{1,1}$&$H_{1,1}$          \\
\hline
$\alpha_{5,1}$&$\alpha_{6,1}$&$\alpha_{7,1}$&$\alpha_{8,2}$&
$\alpha_{1,1}$&$\alpha_{2,1}$&$\alpha_{3,1}$&$\alpha_{4,1}$&
$\alpha_{5,1}$&$\alpha_{6,1}$             \\
\hline
$6A_1\oplus D_8$& $6A_1\oplus D_8$&$6A_1\oplus D_8$&$8A_1\oplus D_7$&
$D_8$&$D_8$&$D_8$&$D_8$&$D_8$&$D_8$  \\
\hline
\end{tabular}
\end{center}

\begin{center}
\begin{tabular}{|c|c|c|c|c|c|c|c|c|c|c|}
\hline
$9$&$9$&$11$&$11$&$11$&$11$           \\
\hline
$H_{1,1}$&$H_{1,1}$&$H_{1,1}$&$H_{1,1}$&$H_{1,1}$&$H_{1,1}$         \\
\hline
$\alpha_{7,1}$&$\alpha_{8,1}$&$\alpha_{1,1}$&$\alpha_{2,1}$&$\alpha_{3,1}$&$\alpha_{4,1}$             \\
\hline
$D_8$&$D_8$&$4A_1\oplus D_4\oplus D_6$&$4A_1\oplus D_4\oplus D_6$&$4A_1\oplus D_4\oplus D_6$&
$4A_1\oplus D_4\oplus D_6$        \\
\hline
\end{tabular}
\end{center}

\begin{center}
\begin{tabular}{|c|c|c|c|c|c|c|c|c|c|c|}
\hline
$11$&$11$&$11$&$11$&$12$            \\
\hline
$H_{1,1}$&$H_{1,1}$&$H_{1,1}$&$H_{1,1}$&$H_{1,1}$                \\
\hline
$\alpha_{5,1}$&$\alpha_{6,2}$&$\alpha_{1,3}$&$\alpha_{3,3}$&
$\alpha_{1,1}$,$\alpha_{1,2}$,$\alpha_{1,3}$,$\alpha_{1,4}$            \\
\hline
$4A_1\oplus D_4\oplus D_6$&$6A_1\oplus D_4\oplus D_5$&$6A_1\oplus A_3\oplus D_6$&
$6A_1\oplus A_3\oplus D_6$&$A_3\oplus 3D_4$        \\
\hline
\end{tabular}
\end{center}

\begin{center}
\begin{tabular}{|c|c|c|c|c|c|c|c|c|c|c|}
\hline
$12$&$12$&$12$&$12$&$12$&$12$&$12$                          \\
\hline
$H_{1,1}$&$H_{1,2}$&$H_{1,2}$&$H_{1,2}$&$H_{1,2}$&$H_{1,2}$&$H_{1,2}$                   \\
\hline
$\alpha_{3,1}$,$\alpha_{3,2}$,$\alpha_{3,3}$,$\alpha_{3,4}$&
$\alpha_{1,1}$&$\alpha_{1,3}$,$\alpha_{6,3}$&$\alpha_{3,1}$&$\alpha_{2,3}$&
$\alpha_{3,3}$,$\alpha_{5,3}$&$\alpha_{4,3}$             \\
\hline
$A_3\oplus 3D_4$&$A_3\oplus E_6$&$D_4\oplus E_6$&$A_3\oplus D_6$&
$D_4\oplus E_6$&$D_4\oplus E_6$&$D_4\oplus E_6$         \\
\hline
\end{tabular}
\end{center}

\begin{center}
\begin{tabular}{|c|c|c|c|c|c|c|c|c|c|c|}
\hline
$13$&$13$&$13$&$13$&$14$&$14$&$14$&$14$&$14$           \\
\hline
$H_{1,1}$&$H_{1,1}$&$H_{1,1}$&$H_{1,1}$&
$H_{1,1}$&$H_{1,1}$&$H_{1,1}$&$H_{1,1}$&$H_{1,1}$                \\
\hline
$\alpha_{1,1}$,$\alpha_{1,2}$&$\alpha_{2,1}$,$\alpha_{2,2}$&$\alpha_{3,1}$,$\alpha_{3,2}$&
$\alpha_{4,1}$,$\alpha_{4,2}$&
$\alpha_{5,1}$,$\alpha_{5,2}$&$\alpha_{1,3}$&$\alpha_{2,3}$&$\alpha_{3,3}$&$\alpha_{4,3}$             \\
\hline
$8A_1\oplus D_6$&$8A_1\oplus D_6$&$8A_1\oplus D_6$&$8A_1\oplus D_6$&
$D_4\oplus D_5$&$2D_5$&$2D_5$&$2D_5$&$2D_5$        \\
\hline
\end{tabular}
\end{center}

\begin{center}
\begin{tabular}{|c|c|c|c|c|c|c|c|c|c|c|}
\hline
$14$&$15$&$15$&$15$&$15$&$16$&$16$&$16$      \\
\hline
$H_{1,1}$&$H_{1,1}$&$H_{1,1}$&$H_{1,1}$&$H_{1,1}$&
$H_{1,1}$&$H_{1,1}$&$H_{1,1}$                                    \\
\hline
$\alpha_{5,3}$,$\alpha_{6,3}$&$\alpha_{1,2}$,$\alpha_{8,2}$&$\alpha_{2,2}$,$\alpha_{7,2}$&
$\alpha_{3,2}$,$\alpha_{6,2}$&$\alpha_{4,2}$,$\alpha_{5,2}$&
$\alpha_{1,2}$&$\alpha_{2,2}$&$\alpha_{3,2}$           \\
\hline
$2D_5$&$A_8$&$A_8$&$A_8$&$A_8$&
$2A_1\oplus A_7$&$2A_1\oplus A_7$&$2A_1\oplus A_7$             \\
\hline
\end{tabular}
\end{center}

\begin{center}
\begin{tabular}{|c|c|c|c|c|c|c|c|c|c|c|}
\hline
$16$&$16$&$16$&$16$&
$16$&$16$&$16$            \\
\hline
$H_{1,1}$&$H_{1,1}$&$H_{1,1}$&$H_{1,1}$&
$H_{1,2}$&$H_{1,2}$&$H_{1,2}$                \\
\hline
$\alpha_{1,3}$&$\alpha_{2,3}$&$\alpha_{3,3}$&$\alpha_{4,3}$,$\alpha_{5,3}$&
$\alpha_{1,1}$,$\alpha_{1,2}$&$\alpha_{2,1}$,$\alpha_{2,2}$&
$\alpha_{3,1}$,$\alpha_{3,2}$         \\
\hline
$4A_1\oplus A_7$&$4A_1\oplus A_7$&$4A_1\oplus A_7$&$4A_1\oplus A_7$&
$6A_1\oplus 2D_4$&$6A_1\oplus 2D_4$&$6A_1\oplus 2D_4$         \\
\hline
\end{tabular}
\end{center}

\begin{center}
\begin{tabular}{|c|c|c|c|c|c|c|c|c|c|c|}
\hline
$16$&$16$&$16$&$16$&$16$&$16$&$18$                        \\
\hline
$H_{1,2}$&$H_{1,3}$&$H_{1,3}$&$H_{1,3}$&$H_{1,3}$&$H_{1,3}$&
$H_{1,1}$     \\
\hline
$\alpha_{4,3}$,$\alpha_{4,4}$&
$\alpha_{1,1}$,$\alpha_{7,1}$&$\alpha_{2,1}$,$\alpha_{6,1}$&
$\alpha_{3,1}$,$\alpha_{5,1}$&$\alpha_{4,1}$&$\alpha_{4,4}$&$\alpha_{1,2}$              \\
\hline
$8A_1\oplus A_3\oplus D_4$&$D_4\oplus D_5$&$D_4\oplus D_5$&$D_4\oplus D_5$&$D_4\oplus D_5$&
$A_3\oplus D_5$&$A_1\oplus A_3\oplus A_5$            \\
\hline
\end{tabular}
\end{center}

\begin{center}
\begin{tabular}{|c|c|c|c|c|c|c|c|c|c|c|}
\hline
$18$&$18$&$18$&$18$&$18$            \\
\hline
$H_{1,1}$&$H_{1,1}$&$H_{1,1}$&$H_{1,1}$&$H_{1,1}$      \\
\hline
$\alpha_{2,2}$&$\alpha_{1,3}$,$\alpha_{5,3}$&$\alpha_{2,3}$,$\alpha_{4,3}$&
$\alpha_{3,3}$&$\alpha_{1,5}$ \\
\hline
$A_1\oplus A_3\oplus A_5$&$3A_1\oplus A_3\oplus A_5$&$3A_1\oplus A_3\oplus A_5$&
$3A_1\oplus A_3\oplus A_5$&$5A_1\oplus A_5$       \\
\hline
\end{tabular}
\end{center}

\begin{center}
\begin{tabular}{|c|c|c|c|c|c|c|c|c|c|c|}
\hline
$18$&$18$&$19$&$19$&$19$           \\
\hline
$H_{1,2}$&$H_{1,2}$&
$H_{1,1}$&$H_{1,1}$&$H_{1,1}$      \\
\hline
$\alpha_{1,1}$,$\alpha_{1,2}$,$\alpha_{1,3}$,$\alpha_{1,4}$&
$\alpha_{2,1}$,$\alpha_{2,2}$,$\alpha_{2,3}$,$\alpha_{2,4}$&
$\alpha_{1,1}$,$\alpha_{3,1}$,$\alpha_{4,1}$&$\alpha_{2,1}$&
$\alpha_{1,3}$,$\alpha_{1,4}$,$\alpha_{1,5}$,$\alpha_{1,6}$             \\
\hline
$10A_1\oplus D_4$&$10A_1\oplus D_4$&
$4A_3$&$4A_3$&$2A_1\oplus 3A_3$        \\
\hline
\end{tabular}
\end{center}

\begin{center}
\begin{tabular}{|c|c|c|c|c|c|c|c|c|c|c|}
\hline
$19$&$19$&$20$&$20$           \\
\hline
$H_{1,2}$&$H_{1,2}$&$H_{1,1}$&$H_{1,1}$                \\
\hline
$\alpha_{1,1}$,$\alpha_{3,1}$,$\alpha_{4,1}$,$\alpha_{1,5}$,$\alpha_{3,5}$,$\alpha_{4,5}$&
$\alpha_{2,1}$,$\alpha_{2,5}$&
$\alpha_{1,3}$,$\alpha_{4,3}$,$\alpha_{1,4}$,$\alpha_{4,4}$&
$\alpha_{2,3}$,$\alpha_{3,3}$,$\alpha_{2,4}$,$\alpha_{3,4}$            \\
\hline
$2D_4$&$2D_4$&$2A_4$&$2A_4$      \\
\hline
\end{tabular}
\end{center}

\begin{center}
\begin{tabular}{|c|c|c|c|c|c|c|c|c|c|c|}
\hline
$21$&$21$&$21$&$21$            \\
\hline
$H_{1,1}$&$H_{1,2}$&$H_{1,2}$&$H_{1,2}$           \\
\hline
$\alpha_{1,1}$,$\alpha_{1,2}$,$\alpha_{1,3}$,$\alpha_{1,4}$,$\alpha_{1,5}$,$\alpha_{1,6}$,$\alpha_{1,7}$,$\alpha_{1,8}$&
$\alpha_{1,2}$,$\alpha_{1,4}$&$\alpha_{1,3}$,$\alpha_{3,3}$,$\alpha_{1,7}$,$\alpha_{3,7}$&
$\alpha_{2,3}$,$\alpha_{2,7}$  \\
\hline
$14A_1$&$2A_1\oplus 2A_3$&$4A_1\oplus 2A_3$&$4A_1\oplus 2A_3$       \\
\hline
\end{tabular}
\end{center}

\begin{center}
\begin{tabular}{|c|c|c|c|c|c|c|c|c|c|c|}
\hline
$22$&$23$ $\ast$                     \\
\hline
$H_{1,1}$&$H_{1,1}$                 \\
\hline
$\alpha_{1,3}$,$\alpha_{2,3}$,$\alpha_{1,6}$,$\alpha_{2,6}$,$\alpha_{1,7}$,$\alpha_{2,7}$,$\alpha_{1,9}$,$\alpha_{2,9}$&
$\alpha_2$,$\alpha_3$,$\alpha_5$,$\alpha_6$,$\alpha_7$,$\alpha_{10}$,$\alpha_{11}$,$\alpha_{14}$\\
\hline
$4A_2$&$8A_1$       \\
\hline
\end{tabular}
\end{center}

\medskip


{\bf n}=2, degeneration $\aaa_1$:

\begin{center}
\begin{tabular}{|c|c|c|c|c|c|c|c|c|c|c|}
\hline
$j$&$12$&$12$&$12$&$12$&$14$&$14$&$14$&$14$            \\
\hline
$H$&$H_{2,1}$&$H_{2,1}$&$H_{2,1}$&$H_{2,1}$&
$H_{2,1}$&$H_{2,1}$&$H_{2,1}$&$H_{2,1}$                \\
\hline
orbit of&$\alpha_{1,1}$,$\alpha_{6,1}$&$\alpha_{2,1}$&$\alpha_{3,1}$,$\alpha_{5,1}$&$\alpha_{4,1}$&
$\alpha_{1,1}$&$\alpha_{2,1}$&$\alpha_{3,1}$&$\alpha_{4,1}$             \\
\hline
$(S^\perp_{N_j})^{(2)}$&$A_5$&$A_5$&$A_5$&$A_5$&
$A_1\oplus D_4$&$A_1\oplus D_4$&$A_1\oplus D_4$&$A_1\oplus D_4$        \\
\hline
\end{tabular}
\end{center}

\begin{center}
\begin{tabular}{|c|c|c|c|c|c|c|c|c|c|c|}
\hline
$14$&$17$&$17$&$17$&$18$&$18$&$18$&$18$                 \\
\hline
$H_{2,1}$&$H_{2,1}$&$H_{2,1}$&$H_{2,1}$&
$H_{2,1}$&$H_{2,1}$&$H_{2,1}$&$H_{2,1}$              \\
\hline
$\alpha_{5,1}$,$\alpha_{6,1}$&
$\alpha_{1,1}$,$\alpha_{6,1}$&$\alpha_{2,1}$,$\alpha_{5,1}$&$\alpha_{3,1}$,$\alpha_{4,1}$&
$\alpha_{1,1}$,$\alpha_{5,1}$&$\alpha_{2,1}$,$\alpha_{4,1}$&$\alpha_{3,1}$&$\alpha_{2,5}$             \\
\hline
$A_1\oplus D_4$&$A_4$&$A_4$&$A_4$&$A_2\oplus A_3$&$A_2\oplus A_3$&$A_2\oplus A_3$&$A_5$    \\
\hline
\end{tabular}
\end{center}

\begin{center}
\begin{tabular}{|c|c|c|c|c|c|c|}
\hline
$19$&$19$&$19$&$19$&$21$            \\
\hline
$H_{2,1}$&$H_{2,2}$&$H_{2,2}$&$H_{2,2}$&$H_{2,1}$                           \\
\hline
$\alpha_{2,1}$,$\alpha_{2,2}$,$\alpha_{2,3}$,$\alpha_{2,4}$,$\alpha_{2,5}$,$\alpha_{2,6}$&
$\alpha_{2,3}$,$\alpha_{2,5}$&$\alpha_{1,4}$,$\alpha_{3,4}$,$\alpha_{4,4}$&$\alpha_{2,4}$&
$\alpha_{1,1}$,$\alpha_{3,1}$,$\alpha_{1,2}$,$\alpha_{3,2}$ \\
\hline
$5A_2$&$A_2\oplus D_4$&$3A_1\oplus 2A_2$&$3A_1\oplus 2A_2$&$A_1\oplus A_3$        \\
\hline
\end{tabular}
\end{center}

\begin{center}
\begin{tabular}{|c|c|c|c|c|c|c|}
\hline
$21$&$22$&$23$ $\ast$               \\
\hline
$H_{2,1}$&$H_{2,1}$&$H_{2,1}$           \\
\hline
$\alpha_{2,1}$,$\alpha_{2,2}$&
$\alpha_{1,6}$,$\alpha_{2,6}$,$\alpha_{1,8}$,$\alpha_{2,8}$,$\alpha_{1,10}$,$\alpha_{2,10}$&
$\alpha_3$,$\alpha_4$,$\alpha_{14}$,$\alpha_{17}$,$\alpha_{21}$,$\alpha_{24}$             \\
\hline
$A_1\oplus A_3$&$2A_2$&$5A_1$       \\
\hline
\end{tabular}
\end{center}

\medskip

{\bf n}=2, degeneration $3\aaa_1$:

\begin{center}
\begin{tabular}{|c|c|c|c|c|c|c|c|c|c|}
\hline
$j$&$12$&$12$&$12$&$12$&$14$&$14$&$14$&$14$&$14$            \\
\hline
$H$&$H_{2,1}$&$H_{2,1}$&$H_{2,1}$&$H_{2,1}$&
$H_{2,1}$&$H_{2,1}$&$H_{2,1}$&$H_{2,1}$&$H_{2,1}$                \\
\hline
orbit of&$\alpha_{1,2}$,$\alpha_{6,2}$&$\alpha_{2,2}$&$\alpha_{3,2}$,$\alpha_{5,2}$&$\alpha_{4,2}$&
$\alpha_{1,2}$&$\alpha_{2,2}$&$\alpha_{3,2}$&$\alpha_{4,2}$&$\alpha_{5,2}$,$\alpha_{6,2}$             \\
\hline
$(S^\perp_{N_j})^{(2)}$&$E_6$&$E_6$&$E_6$&$E_6$&$D_6$&$D_6$&$D_6$&$D_6$&$D_6$         \\
\hline
\end{tabular}
\end{center}

\begin{center}
\begin{tabular}{|c|c|c|c|c|c|c|}
\hline
$17$&$17$&$17$&$18$&$18$&$18$&$18$         \\
\hline
$H_{2,1}$&$H_{2,1}$&$H_{2,1}$&
$H_{2,1}$&$H_{2,1}$&$H_{2,1}$&$H_{2,1}$                \\
\hline
$\alpha_{1,2}$,$\alpha_{6,2}$&$\alpha_{2,2}$,$\alpha_{5,2}$&$\alpha_{3,2}$,$\alpha_{4,2}$&
$\alpha_{1,2}$,$\alpha_{5,2}$&$\alpha_{2,2}$,$\alpha_{4,2}$&$\alpha_{3,2}$&$\alpha_{1,5}$             \\
\hline
$A_6$&$A_6$&$A_6$&$A_2\oplus A_5$&$A_2\oplus A_5$&$A_2\oplus A_5$&$A_1\oplus A_5$        \\
\hline
\end{tabular}
\end{center}

\begin{center}
\begin{tabular}{|c|c|c|c|c|c|c|c|c|}
\hline
$19$&$19$&$19$&$19$&$21$         \\
\hline
$H_{2,1}$&$H_{2,2}$&$H_{2,2}$&$H_{2,2}$&$H_{2,1}$               \\
\hline
$\alpha_{1,1}$,$\alpha_{1,2}$,$\alpha_{1,3}$,$\alpha_{1,4}$,$\alpha_{1,5}$,$\alpha_{1,6}$&
$\alpha_{1,1}$,$\alpha_{3,1}$,$\alpha_{4,1}$&$\alpha_{2,1}$&$\alpha_{1,3}$,$\alpha_{1,5}$&
$\alpha_{1,3}$,$\alpha_{3,3}$,$\alpha_{1,5}$,$\alpha_{3,5}$          \\
\hline
$A_1\oplus 5A_2$&$2A_2\oplus D_4$&$2A_2\oplus D_4$&$A_1\oplus A_2\oplus D_4$&
$2A_3$         \\
\hline
\end{tabular}
\end{center}

\begin{center}
\begin{tabular}{|c|c|c|c|c|c|c|}
\hline
$21$&$22$&$23$ $\ast$                           \\
\hline
$H_{2,1}$&$H_{2,1}$&$H_{2,1}$                        \\
\hline
$\alpha_{2,3}$,$\alpha_{2,5}$&
$\alpha_{1,1}$,$\alpha_{2,1}$,$\alpha_{1,2}$,$\alpha_{2,2}$,$\alpha_{1,7}$,$\alpha_{2,7}$&
$\alpha_1$,$\alpha_2$,$\alpha_5$,$\alpha_6$,$\alpha_{12}$,$\alpha_{15}$      \\
\hline
$2A_3$&$3A_2$&$6A_1$        \\
\hline
\end{tabular}
\end{center}

\medskip

{\bf n}=3, degeneration $\aaa_1$:

\begin{center}
\begin{tabular}{|c|c|c|c|c|c|c|c|c|}
\hline
$j$&$12$&$12$&$16$&$16$&$16$&$16$&$18$            \\
\hline
$H$&$H_{3,1}$&$H_{3,1}$&$H_{3,1}$&$H_{3,2}$&$H_{3,2}$&$H_{3,2}$&$H_{3,1}$                \\
\hline
orbit of&$\alpha_{2,1}$,$\alpha_{2,2}$&$\alpha_{4,1}$,$\alpha_{4,2}$&$\alpha_{4,1}$,$\alpha_{4,2}$&
$\alpha_{1,3}$,$\alpha_{1,4}$&$\alpha_{2,3}$,$\alpha_{2,4}$&$\alpha_{3,3}$,$\alpha_{3,4}$&
$\alpha_{3,1}$,$\alpha_{3,2}$             \\
\hline
$(S^\perp_{N_j})^{(2)}$&$3A_1\oplus D_4$&$3A_1\oplus D_4$&$7A_1$&
$3A_1\oplus D_4$&$3A_1\oplus D_4$&$3A_1\oplus D_4$&$5A_1\oplus A_3$        \\
\hline
\end{tabular}
\end{center}

\begin{center}
\begin{tabular}{|c|c|c|c|c|c|c|}
\hline
$18$&$18$&$19$&$19$&$19$&$19$&$21$            \\
\hline
$H_{3,1}$&$H_{3,1}$&$H_{3,1}$&$H_{3,1}$&$H_{3,2}$&$H_{3,2}$&$H_{3,1}$                \\
\hline
$\alpha_{2,5}$&$\alpha_{4,5}$&$\alpha_{2,5}$,$\alpha_{2,6}$&$\alpha_{4,5}$,$\alpha_{4,6}$&
$\alpha_{1,3}$,$\alpha_{3,3}$,$\alpha_{4,3}$,$\alpha_{1,4}$,$\alpha_{3,4}$,$\alpha_{4,4}$&
$\alpha_{2,3}$,$\alpha_{2,4}$&$\alpha_{2,1}$,$\alpha_{2,2}$,$\alpha_{2,4}$,$\alpha_{2,5}$             \\
\hline
$7A_1$&$7A_1$&$A_1\oplus A_3$&$A_1\oplus A_3$&$3A_1\oplus D_4$&$3A_1\oplus D_4$&$7A_1$         \\
\hline
\end{tabular}
\end{center}

\begin{center}
\begin{tabular}{|c|c|c|c|c|c|c|c|c|c|}
\hline
$21$&$21$&$21$&$21$&$21$&$21$&$22$ $\ast$          \\
\hline
$H_{3,2}$&$H_{3,2}$&$H_{3,2}$&$H_{3,3}$&$H_{3,3}$&$H_{3,4}$&$H_{3,1}$                \\
\hline
$\alpha_{1,1}$,$\alpha_{3,1}$&$\alpha_{2,1}$&$\alpha_{2,2}$,$\alpha_{2,4}$,$\alpha_{2,5}$&
$\alpha_{1,1}$,$\alpha_{3,1}$,$\alpha_{1,5}$,$\alpha_{3,5}$&$\alpha_{2,1}$,$\alpha_{2,5}$&
$\alpha_{2,1}$,$\alpha_{2,5}$&
$\alpha_{1,1}$,$\alpha_{2,1}$,$\alpha_{1,2}$,$\alpha_{2,2}$             \\
\hline
$7A_1$&$7A_1$&$5A_1\oplus A_3$&$A_1\oplus A_3$&$A_1\oplus A_3$&$3A_1$&$A_2$        \\
\hline
\end{tabular}
\end{center}

\begin{center}
\begin{tabular}{|c|c|c|c|c|c|c|}
\hline
$23$&$23$            \\
\hline
$H_{3,1}$&$H_{3,2}$                 \\
\hline
$\alpha_1$,$\alpha_5$,$\alpha_6$,$\alpha_{10}$,$\alpha_{15}$,$\alpha_{17}$,
$\alpha_{21}$,$\alpha_{24}$&$\alpha_1$,$\alpha_5$,$\alpha_6$,$\alpha_{21}$                  \\
\hline
$7A_1$&$3A_1$           \\
\hline
\end{tabular}
\end{center}

\medskip

{\bf n}=3, degeneration $2\aaa_1$:

\begin{center}
\begin{tabular}{|c|c|c|c|c|c|c|c|c|c|c|c|c|c|}
\hline
$j$&$12$&$12$&$12$&$12$&$16$&$16$&$16$&$16$&$16$            \\
\hline
$H$&$H_{3,1}$&$H_{3,1}$&$H_{3,1}$&$H_{3,1}$&
$H_{3,1}$&$H_{3,1}$&$H_{3,1}$&$H_{3,1}$&$H_{3,1}$                \\
\hline
orbit of&$\alpha_{1,1}$,$\alpha_{1,2}$&$\alpha_{3,1}$,$\alpha_{3,2}$&$\alpha_{2,3}$&$\alpha_{4,3}$&
$\alpha_{1,1}$,$\alpha_{1,2}$&$\alpha_{2,1}$,$\alpha_{2,2}$&$\alpha_{3,1}$,$\alpha_{3,2}$&$\alpha_{1,3}$&
$\alpha_{2,3}$             \\
\hline
$(S^\perp_{N_j})^{(2)}$&$A_3\oplus D_4$&$A_3\oplus D_4$&$2D_4$&$2D_4$&
$6A_1$&$6A_1$&$6A_1$&$8A_1$&$8A_1$        \\
\hline
\end{tabular}
\end{center}

\begin{center}
\begin{tabular}{|c|c|c|c|c|c|c|c|c|c|c|c|}
\hline
$16$&$16$&$16$&$18$&$18$&$18$ $\ast$ &$18$&$19$&$19$            \\
\hline
$H_{3,1}$&$H_{3,2}$&$H_{3,2}$&$H_{3,1}$&$H_{3,1}$&$H_{3,1}$&$H_{3,1}$&
$H_{3,1}$&$H_{3,1}$         \\
\hline
$\alpha_{3,3}$&$\alpha_{4,1}$&$\alpha_{4,3}$,$\alpha_{4,4}$&
$\alpha_{1,1}$,$\alpha_{1,2}$&$\alpha_{2,1}$,$\alpha_{2,2}$&$\alpha_{3,3}$&$\alpha_{1,5}$&
$\alpha_{2,1}$,$\alpha_{2,3}$&$\alpha_{4,1}$,$\alpha_{4,3}$    \\
\hline
$8A_1$&$2D_4$&$A_3\oplus D_4$&$4A_1\oplus A_3$&$4A_1\oplus A_3$&$6A_1\oplus A_3$&$8A_1$&
$2A_3$&$2A_3$          \\
\hline
\end{tabular}
\end{center}

\begin{center}
\begin{tabular}{|c|c|c|c|c|c|c|}
\hline
$19$&$19$&$19$&$21$            \\
\hline
$H_{3,1}$&$H_{3,3}$&$H_{3,3}$&$H_{3,1}$                       \\
\hline
$\alpha_{1,5}$,$\alpha_{1,6}$&$\alpha_{1,1}$,$\alpha_{3,1}$,$\alpha_{4,1}$,
$\alpha_{1,3}$,$\alpha_{3,3}$,$\alpha_{4,3}$,$\alpha_{1,5}$,$\alpha_{3,5}$,$\alpha_{4,5}$&
$\alpha_{2,1}$,$\alpha_{2,3}$,$\alpha_{2,5}$&
$\alpha_{1,1}$,$\alpha_{1,2}$,$\alpha_{1,4}$,$\alpha_{1,5}$             \\
\hline
$2A_1\oplus A_3$&$\{ 0\}$&$\{ 0\}$&$6A_1$        \\
\hline
\end{tabular}
\end{center}

\begin{center}
\begin{tabular}{|c|c|c|c|c|c|c|}
\hline
$21$&$21$&$21$&$21$&$21$&$21$&$21$                             \\
\hline
$H_{3,1}$&$H_{3,2}$&$H_{3,3}$&$H_{3,4}$&$H_{3,4}$&$H_{3,4}$&$H_{3,4}$                    \\
\hline
$\alpha_{2,3}$,$\alpha_{2,7}$&$\alpha_{1,2}$,$\alpha_{1,4}$,$\alpha_{1,5}$&
$\alpha_{2,2}$,$\alpha_{2,3}$,$\alpha_{2,7}$&
$\alpha_{1,1}$,$\alpha_{1,5}$&$\alpha_{2,2}$&$\alpha_{1,3}$,$\alpha_{3,3}$,
$\alpha_{1,7}$,$\alpha_{3,7}$&$\alpha_{2,3}$,$\alpha_{2,7}$             \\
\hline
$8A_1$&$4A_1\oplus A_3$&$2A_3$&$2A_1$&$4A_1$&$4A_1$&$4A_1$         \\
\hline
\end{tabular}
\end{center}

\begin{center}
\begin{tabular}{|c|c|c|c|c|c|c|}
\hline
$22$&$23$&$23$                       \\
\hline
$H_{3,1}$&$H_{3,2}$&$H_{3,3}$                  \\
\hline
$\alpha_{1,3}$,$\alpha_{2,3}$,$\alpha_{1,5}$,$\alpha_{2,5}$,$\alpha_{1,7}$,$\alpha_{2,7}$&
$\alpha_8$,$\alpha_{10}$,$\alpha_{11}$,$\alpha_{12}$,$\alpha_{15}$,$\alpha_{16}$&
$\alpha_1$,$\alpha_2$,$\alpha_3$,$\alpha_4$,$\alpha_5$,$\alpha_6$,$\alpha_7$,$\alpha_8$,$\alpha_{11}$,$\alpha_{12}$,$\alpha_{15}$,$\alpha_{16}$             \\
\hline
$2A_2$&$4A_1$&$\{ 0\}$        \\
\hline
\end{tabular}
\end{center}

\medskip

{\bf n}=3, degeneration $4\aaa_1$:

\begin{center}
\begin{tabular}{|c|c|c|c|c|c|c|c|c|c|c|c|}
\hline
$j$&$12$&$12$&$16$&$16$&$16$&$16$&$18$&$18$&$19$&$19$     \\
\hline
$H$&$H_{3,1}$&$H_{3,1}$&$H_{3,1}$&$H_{3,2}$&$H_{3,2}$&$H_{3,2}$&
$H_{3,1}$&$H_{3,1}$&$H_{3,1}$&$H_{3,2}$            \\
\hline
orbit of&$\alpha_{1,3}$&$\alpha_{3,3}$&$\alpha_{4,3}$&
$\alpha_{1,1}$&$\alpha_{2,1}$&$\alpha_{3,1}$&$\alpha_{1,3}$&$\alpha_{2,3}$&
$\alpha_{1,1}$,$\alpha_{1,3}$&$\alpha_{2,1}$             \\
\hline
$(S^\perp_{N_j})^{(2)}$&$2D_4$&$2D_4$&$8A_1$&$2D_4$&$2D_4$&$2D_4$&$6A_1\oplus A_3$&
$6A_1\oplus A_3$&$2A_3$&$2D_4$       \\
\hline
\end{tabular}
\end{center}

\begin{center}
\begin{tabular}{|c|c|c|c|c|c|c|c|c|c|}
\hline
$19$&$21$&$21$&$21$&$21$&$21$&$22$ $\ast$                           \\
\hline
$H_{3,2}$&$H_{3,1}$&$H_{3,2}$&$H_{3,2}$&$H_{3,3}$&$H_{3,4}$&$H_{3,1}$                                  \\
\hline
$\alpha_{1,1}$,$\alpha_{3,1}$,$\alpha_{4,1}$&$\alpha_{1,3}$,$\alpha_{1,7}$&
$\alpha_{1,3}$,$\alpha_{3,3}$&$\alpha_{2,3}$&$\alpha_{1,2}$,
$\alpha_{1,3}$,$\alpha_{1,7}$&$\alpha_{1,2}$&$\alpha_{1,6}$,$\alpha_{2,6}$             \\
\hline
$2D_4$&$6A_1$&$6A_1\oplus A_3$&$6A_1\oplus A_3$&$2A_3$&$4A_1$&$2A_2$         \\
\hline
\end{tabular}
\end{center}

\begin{center}
\begin{tabular}{|c|c|c|c|c|c|c|}
\hline
$23$&$23$              \\
\hline
$H_{3,1}$&$H_{3,2}$               \\
\hline
$\alpha_2$,$\alpha_3$,$\alpha_4$,$\alpha_7$&$\alpha_2$,$\alpha_4$            \\
\hline
$8A_1$&$4A_1$        \\
\hline
\end{tabular}
\end{center}

\medskip

{\bf n}=4, degeneration $\aaa_1 $:

\begin{center}
\begin{tabular}{|c|c|c|c|c|c|c|c|c|c|}
\hline
$j$&$13$&$13$&$13$&$13$&$18$&$18$&$18$&$19$            \\
\hline
$H$&$H_{4,1}$&$H_{4,1}$&$H_{4,1}$&$H_{4,1}$&$H_{4,1}$&$H_{4,1}$&$H_{4,1}$&$H_{4,1}$                \\
\hline
orbit of&$\alpha_{1,3}$&$\alpha_{2,3}$&$\alpha_{3,3}$&$\alpha_{4,3}$&
$\alpha_{1,5}$,$\alpha_{3,5}$&$\alpha_{2,5}$&$\alpha_{4,5}$&$\alpha_{2,3}$,$\alpha_{2,4}$             \\
\hline
$(S^\perp_{N_j})^{(2)}$&$A_1\oplus A_3$&$A_1\oplus A_3$&$A_1\oplus A_3$&$A_1\oplus A_3$&
$3A_1$&$3A_1$&$3A_1$&$A_1\oplus A_3$ \\
\hline
\end{tabular}
\end{center}

\begin{center}
\begin{tabular}{|c|c|c|c|c|c|c|c|}
\hline
$19$&$20$&$20$&$21$&$21$&$21$&$22$            \\
\hline
$H_{4,1}$&$H_{4,1}$&$H_{4,1}$&$H_{4,1}$&$H_{4,1}$&$H_{4,1}$&$H_{4,1}$                \\
\hline
$\alpha_{3,3}$,$\alpha_{1,4}$&$\alpha_{1,1}$,$\alpha_{4,1}$&
$\alpha_{2,1}$,$\alpha_{3,1}$&$\alpha_{2,1}$&$\alpha_{1,2}$,$\alpha_{3,2}$&
$\alpha_{2,2}$&$\alpha_{1,2}$,$\alpha_{2,2}$,$\alpha_{1,5}$,$\alpha_{2,5}$   \\
\hline
$A_1\oplus A_3$&$2A_1\oplus A_2$&$2A_1\oplus A_2$&
$A_1\oplus A_3$&$3A_1$&$3A_1$&$2A_1\oplus A_2$        \\
\hline
\end{tabular}
\end{center}

\begin{center}
\begin{tabular}{|c|c|c|c|c|c|c|}
\hline
$23$ \\
\hline
$H_{4,1}$       \\
\hline
$\alpha_1$,$\alpha_4$,$\alpha_{17}$,$\alpha_{18}$     \\
\hline
$3A_1$        \\
\hline
\end{tabular}
\end{center}

\medskip

\medskip

{\bf n}=4, degeneration $2\aaa_1$:

\begin{center}
\begin{tabular}{|c|c|c|c|c|c|c|c|c|}
\hline
$j$&$13$&$13$&$18$&$19$&$20$&$21$&$21$&$23$ $\ast$          \\
\hline
$H$&$H_{4,1}$&$H_{4,1}$&$H_{4,1}$&$H_{4,1}$&$H_{4,1}$&$H_{4,1}$&$H_{4,1}$&$H_{4,1}$           \\
\hline
orbit of&$\alpha_{5,1}$&$\alpha_{5,3}$&$\alpha_{3,1}$,$\alpha_{3,3}$&$\alpha_{1,3}$,$\alpha_{3,4}$&
$\alpha_{1,2}$&$\alpha_{1,1}$&$\alpha_{2,3}$&$\alpha_8$,$\alpha_{12}$             \\
\hline
$(S^\perp_{N_j})^{(2)}$&$D_5$&$D_4$&$D_4$&$2A_1\oplus A_3$&$A_4$&
$A_3$&$2A_1\oplus A_3$&$4A_1$\\
\hline
\end{tabular}
\end{center}

\medskip

{\bf n}=4, degeneration $4\aaa_1$:

\begin{center}
\begin{tabular}{|c|c|c|c|c|c|c|c|c|c|c|c|c|c|}
\hline
$j$&$13$&$13$&$13$&$13$&$18$&$18$&$19$&$19$&$19$&$20$           \\
\hline
$H$&$H_{4,1}$&$H_{4,1}$&$H_{4,1}$&$H_{4,1}$&$H_{4,1}$&$H_{4,1}$&
$H_{4,1}$&$H_{4,1}$&$H_{4,1}$&$H_{4,1}$                \\
\hline
orbit of&$\alpha_{1,1}$&$\alpha_{2,1}$&$\alpha_{3,1}$&$\alpha_{4,1}$&
$\alpha_{1,1}$,$\alpha_{1,3}$&$\alpha_{2,1}$,$\alpha_{2,3}$&
$\alpha_{1,1}$,$\alpha_{3,1}$&$\alpha_{2,1}$&$\alpha_{4,1}$&
$\alpha_{1,3}$,$\alpha_{4,3}$             \\
\hline
$(S^\perp_{N_j})^{(2)}$&$D_5$&$D_5$&$D_5$&$D_5$&$D_4$&$D_4$&
$2A_3$&$2A_3$&$2A_3$&$2A_1\oplus A_4$       \\
\hline
\end{tabular}
\end{center}

\begin{center}
\begin{tabular}{|c|c|c|c|c|c|c|}
\hline
$20$&$21$&$21$&$21$&$22$&$23$ $\ast$          \\
\hline
$H_{4,1}$&$H_{4,1}$&$H_{4,1}$&$H_{4,1}$&$H_{4,1}$&$H_{4,1}$      \\
\hline
$\alpha_{2,3}$,$\alpha_{3,3}$&$\alpha_{1,3}$&$\alpha_{1,4}$,$\alpha_{3,4}$&$\alpha_{2,4}$&
$\alpha_{1,3}$,$\alpha_{2,3}$,$\alpha_{1,6}$,$\alpha_{2,6}$&$\alpha_2$,$\alpha_3$,$\alpha_7$,$\alpha_{11}$ \\
\hline
$2A_1\oplus A_4$&$2A_1\oplus A_3$&$2A_1\oplus A_3$&$2A_1\oplus A_3$&$2A_1\oplus 2A_2$&$4A_1$ \\
\hline
\end{tabular}
\end{center}

\medskip

{\bf n}=4, degeneration $\aaa_2 $:

\begin{center}
\begin{tabular}{|c|c|c|c|c|c|c|}
\hline
$j$ &$20$&$22$  $\ast$           \\
\hline
$H$& $H_{4,1}$&$H_{4,1}$                   \\
\hline
orbit of&$\alpha_{2,2}$&$\alpha_{1,1}$,$\alpha_{1,11}$             \\
\hline
$(S^\perp_{N_j})^{(2)}$&$A_1\oplus A_4$&$A_1\oplus 2A_2$   \\
\hline
\end{tabular}
\end{center}

\medskip

{\bf n}=6, degeneration $\aaa_1 $:

\begin{center}
\begin{tabular}{|c|c|c|c|c|c|c|c|c|c|c|c|c|}
\hline
$j$ &$12$&$12$&$12$&$12$&$12$&$12$&$18$&$18$&$18$           \\
\hline
$H$ &$H_{6,1}$&$H_{6,1}$&$H_{6,1}$&$H_{6,1}$&$H_{6,2}$&$H_{6,2}$&
$H_{6,1}$&$H_{6,1}$&$H_{6,1}$                \\
\hline
orbit of &$\alpha_{1,1}$,$\alpha_{6,1}$&$\alpha_{2,1}$&$\alpha_{3,1}$,$\alpha_{5,1}$&$\alpha_{4,1}$&
$\alpha_{2,1}$&$\alpha_{4,1}$&
$\alpha_{1,1}$,$\alpha_{5,1}$&$\alpha_{2,1}$,$\alpha_{4,1}$&$\alpha_{3,1}$                     \\
\hline
$(S^\perp_{N_j})^{(2)}$&$A_5$&$A_5$&$A_5$&$A_5$&$3A_1$&$3A_1$&
$A_2\oplus A_3$&$A_2\oplus A_3$&$A_2\oplus A_3$       \\
\hline
\end{tabular}
\end{center}

\begin{center}
\begin{tabular}{|c|c|c|c|c|c|c|c|c|c|}
\hline
$18$&$18$&$18$&$19$&$19$&$19$&$19$            \\
\hline
$H_{6,1}$&$H_{6,2}$&$H_{6,2}$&$H_{6,1}$&$H_{6,2}$&$H_{6,2}$&$H_{6,2}$                \\
\hline
$\alpha_{2,5}$&$\alpha_{3,1}$&$\alpha_{2,5}$&
$\alpha_{2,3}$,$\alpha_{2,4}$,$\alpha_{2,5}$,$\alpha_{2,6}$&
$\alpha_{2,4}$&$\alpha_{4,4}$&$\alpha_{2,5}$,$\alpha_{2,6}$             \\
\hline
$A_5$&$2A_1\oplus A_2$&$3A_1$&$3A_2$&
$A_1\oplus 2A_2$&$A_1\oplus 2A_2$&$A_2\oplus A_3$       \\
\hline
\end{tabular}
\end{center}

\begin{center}
\begin{tabular}{|c|c|c|c|c|c|c|c|c|}
\hline
$19$&$19$&$21$&$21$&$21$&$22$&$22$ $\ast$           \\
\hline
$H_{6,3}$&$H_{6,3}$&$H_{6,1}$&$H_{6,1}$&$H_{6,1}$&$H_{6,1}$&$H_{6,2}$ \\
\hline
$\alpha_{1,3}$,$\alpha_{3,3}$,$\alpha_{4,3}$&$\alpha_{2,3}$&
$\alpha_{1,1}$,$\alpha_{3,1}$&$\alpha_{2,1}$&$\alpha_{2,2}$&$\alpha_{1,8}$,$\alpha_{2,8}$&
$\alpha_{1,6}$,$\alpha_{2,6}$,$\alpha_{1,8}$,$\alpha_{2,8}$,$\alpha_{1,10}$,$\alpha_{2,10}$ \\
\hline
$3A_1$&$3A_1$&$3A_1$&$3A_1$&$A_1\oplus A_3$&$\{ 0\}$&$2A_2$       \\
\hline
\end{tabular}
\end{center}

\begin{center}
\begin{tabular}{|c|c|c|c|c|c|c|}
\hline
$23$                          \\
\hline
$H_{6,1}$                     \\
\hline
$\alpha_1$,$\alpha_5$,$\alpha_6$,$\alpha_{24}$             \\
\hline
$3A_1$        \\
\hline
\end{tabular}
\end{center}

\medskip

{\bf n}=6, degeneration $2\aaa_1$:

\begin{center}
\begin{tabular}{|c|c|c|c|c|c|c|c|c|c|c|c|}
\hline
$j$&$12$&$12$&$18$&$18$&$19$&$19$&$19$&$21$&$22$&$23$ $\ast$         \\
\hline
$H$&$H_{6,2}$&$H_{6,2}$&$H_{6,2}$&$H_{6,2}$&$H_{6,1}$&$H_{6,2}$&$H_{6,3}$&
$H_{6,1}$&$H_{6,1}$&$H_{6,1}$         \\
\hline
orbit of&$\alpha_{1,1}$&$\alpha_{3,1}$&$\alpha_{1,1}$&$\alpha_{2,1}$&$\alpha_{2,1}$&
$\alpha_{1,4}$&$\alpha_{2,1}$&$\alpha_{1,2}$&$\alpha_{1,6}$,$\alpha_{2,6}$&$\alpha_{16}$     \\
\hline
$(S^\perp_{N_j})^{(2)}$&$A_3$&$A_3$&$A_1\oplus A_2$&$A_1\oplus A_2$&$4A_2$
&$2A_1\oplus 2A_2$&$D_4$&$A_3$&$A_2$&$4A_1$  \\
\hline
\end{tabular}
\end{center}

\medskip

{\bf n}=6, degeneration $3\aaa_1$:

\begin{center}
\begin{tabular}{|c|c|c|c|c|c|c|c|c|c|}
\hline
$j$ &$12$&$12$&$12$&$12$&$12$&$12$&$18$&$18$&$18$            \\
\hline
$H$ &$H_{6,1}$&$H_{6,1}$&$H_{6,2}$&$H_{6,2}$&$H_{6,2}$&$H_{6,2}$&$H_{6,1}$&$H_{6,1}$&$H_{6,2}$                \\
\hline
orbit of &$\alpha_{2,2}$&$\alpha_{4,2}$&$\alpha_{1,2}$,$\alpha_{6,2}$&
$\alpha_{2,2}$&$\alpha_{3,2}$,$\alpha_{5,2}$&$\alpha_{4,2}$&$\alpha_{3,2}$&$\alpha_{1,5}$&$\alpha_{1,2}$,$\alpha_{5,2}$             \\
\hline
$(S^\perp_{N_j})^{(2)}$&$E_6$&$E_6$&$D_4$&$D_4$&$D_4$&$D_4$&$A_2\oplus A_5$&$A_1\oplus A_5$&$3A_1\oplus A_2$        \\
\hline
\end{tabular}
\end{center}

\begin{center}
\begin{tabular}{|c|c|c|c|c|c|c|c|c|c|}
\hline
$18$&$18$&$18$&$19$&$19$&$19$&$19$            \\
\hline
$H_{6,2}$&$H_{6,2}$&$H_{6,2}$&$H_{6,1}$&$H_{6,2}$&$H_{6,2}$&$H_{6,2}$                \\
\hline
$\alpha_{2,2}$,$\alpha_{4,2}$&$\alpha_{3,2}$&$\alpha_{1,5}$&
$\alpha_{1,3}$,$\alpha_{1,4}$,$\alpha_{1,5}$,$\alpha_{1,6}$&
$\alpha_{2,1}$&$\alpha_{4,1}$&$\alpha_{1,5}$,$\alpha_{1,6}$             \\
\hline
$3A_1\oplus A_2$&$3A_1\oplus A_2$&$4A_1$&$A_1\oplus 3A_2$&
$2A_2\oplus A_3$&$2A_2\oplus A_3$&$A_1\oplus A_2\oplus A_3$   \\
\hline
\end{tabular}
\end{center}

\begin{center}
\begin{tabular}{|c|c|c|c|c|c|c|}
\hline
$19$&$19$&$21$&$21$&$21$&$22$           \\
\hline
$H_{6,3}$&$H_{6,3}$&$H_{6,1}$&$H_{6,1}$&$H_{6,1}$&$H_{6,1}$                \\
\hline
$\alpha_{1,4}$,$\alpha_{3,4}$,$\alpha_{4,4}$&$\alpha_{2,4}$&$\alpha_{2,3}$&
$\alpha_{1,5}$,$\alpha_{3,5}$&$\alpha_{2,5}$&$\alpha_{1,1}$,$\alpha_{2,1}$,
$\alpha_{1,2}$,$\alpha_{2,2}$,$\alpha_{1,7}$,$\alpha_{2,7}$             \\
\hline
$D_4$&$D_4$&$2A_1\oplus A_3$&$2A_1\oplus A_3$&$2A_1\oplus A_3$&$A_2$        \\
\hline
\end{tabular}
\end{center}

\begin{center}
\begin{tabular}{|c|c|c|c|c|c|c|}
\hline
$22$ $\ast$  &$23$                       \\
\hline
$H_{6,2}$&$H_{6,1}$                   \\
\hline
$\alpha_{1,1}$,$\alpha_{2,1}$&$\alpha_3$,$\alpha_7$,$\alpha_{10}$,$\alpha_{11}$   \\
\hline
$3A_2$&$4A_1$        \\
\hline
\end{tabular}
\end{center}

\medskip

{\bf n}=6, degeneration $6\aaa_1$:

\begin{center}
\begin{tabular}{|c|c|c|c|c|c|c|c|c|c|c|}
\hline
$j$ &$12$&$12$&$18$&$18$&$19$&$19$&$19$&$21$&$22$            \\
\hline
$H$ &$H_{6,1}$&$H_{6,1}$&$H_{6,1}$&$H_{6,1}$&$H_{6,1}$&$H_{6,2}$&$H_{6,3}$&$H_{6,1}$&$H_{6,2}$               \\
\hline
orbit of &$\alpha_{1,2}$&$\alpha_{3,2}$&$\alpha_{1,2}$&$\alpha_{2,2}$&$\alpha_{1,1}$&$\alpha_{1,1}$&
$\alpha_{1,1}$&$\alpha_{1,3}$&$\alpha_{1,2}$,$\alpha_{2,2}$             \\
\hline
$(S^\perp_{N_j})^{(2)}$&$E_6$&$E_6$&$A_2\oplus A_5$&$A_2\oplus A_5$&$4A_2$&$2A_2\oplus A_3$&$D_4$&
$2A_1\oplus A_3$&$3A_2$        \\
\hline
\end{tabular}
\end{center}

\begin{center}
\begin{tabular}{|c|c|c|c|c|c|c|}
\hline
$23$ $\ast$                 \\
\hline
$H_{6,1}$             \\
\hline
$\alpha_2$            \\
\hline
$4A_1$       \\
\hline
\end{tabular}
\end{center}

\medskip

\medskip

{\bf n}=9, degeneration $2\aaa_1$:

\begin{center}
\begin{tabular}{|c|c|c|c|c|c|c|}
\hline
$j$  & $21$&$21$&$21$&$23$            \\
\hline
$H$  & $H_{9,1}$&$H_{9,2}$&$H_{9,2}$&$H_{9,1}$                \\
\hline
orbit of &$\alpha_{1,1}$,$\alpha_{1,2}$,$\alpha_{1,4}$,$\alpha_{1,5}$&
$\alpha_{1,1}$,$\alpha_{1,5}$&$\alpha_{2,2}$,$\alpha_{2,3}$,$\alpha_{2,7}$&
$\alpha_5$,$\alpha_6$      \\
\hline
$(S^\perp_{N_j})^{(2)}$&$6A_1$&$2A_1$&$4A_1$&$4A_1$        \\
\hline
\end{tabular}
\end{center}

\begin{center}
\begin{tabular}{|c|c|c|c|c|c|c|}
\hline
$23$        & $23$    $\ast$ \\
\hline
$H_{9,3}$   & $H_{9,4}$               \\
\hline
$\alpha_6$,$\alpha_8$,$\alpha_{10}$,$\alpha_{11}$,$\alpha_{12}$,$\alpha_{15}$,$\alpha_{16}$&
$\alpha_1$,$\alpha_6$,$\alpha_8$,$\alpha_{10}$,$\alpha_{11}$,$\alpha_{12}$,$\alpha_{15}$,$\alpha_{16}$   \\
\hline
$2A_1$& $\{ 0\}$       \\
\hline
\end{tabular}
\end{center}

\medskip

{\bf n}=9, degeneration $4\aaa_1$:

\begin{center}
\begin{tabular}{|c|c|c|c|c|c|c|}
\hline
$j$ &$21$ &$21$&$23$&$23$  $\ast$           \\
\hline
$H$ &$H_{9,1}$ &$H_{9,2}$&$H_{9,1}$&$H_{9,4}$                \\
\hline
orbit of &$\alpha_{2,3}$&$\alpha_{1,2}$,$\alpha_{1,3}$,$\alpha_{1,7}$&
$\alpha_2$,$\alpha_3$,$\alpha_4$,$\alpha_7$&$\alpha_2$,$\alpha_4$             \\
\hline
$(S^\perp_{N_j})^{(2)}$&$8A_1$ &$4A_1$&$4A_1$&$\{ 0\}$        \\
\hline
\end{tabular}
\end{center}

\medskip

{\bf n}=9, degeneration $8\aaa_1$:

\begin{center}
\begin{tabular}{|c|c|c|c|c|c|c|}
\hline
$j$ & $21$&$23$&$23$ $\ast$          \\
\hline
$H$ &$H_{9,1}$&$H_{9,2}$&$H_{9,3}$                \\
\hline
orbit of&$\alpha_{1,3}$&$\alpha_2$,$\alpha_4$&$\alpha_2$             \\
\hline
$(S^\perp_{N_j})^{(2)}$&$8A_1$&$8A_1$&$2A_1$   \\
\hline
\end{tabular}
\end{center}

\medskip

{\bf n}=10, degeneration $\aaa_1$:

\begin{center}
\begin{tabular}{|c|c|c|c|c|c|c|c|c|}
\hline
$j$& $18$&$18$&$19$&$19$&$21$&$21$&$21$&$21$        \\
\hline
$H$&$H_{10,1}$&$H_{10,1}$&$H_{10,1}$&$H_{10,1}$&$H_{10,1}$&$H_{10,1}$&$H_{10,2}$&$H_{10,2}$                 \\
\hline
orbit of&$\alpha_{2,5}$&$\alpha_{4,5}$&$\alpha_{1,3}$,$\alpha_{3,4}$&$\alpha_{2,3}$,$\alpha_{2,4}$&
$\alpha_{1,1}$,$\alpha_{3,1}$&$\alpha_{2,1}$&$\alpha_{2,1}$&$\alpha_{2,2}$  \\
\hline
$(S^\perp_{N_j})^{(2)}$&$A_1$&$A_1$&$A_1\oplus A_3$&$A_1\oplus A_3$&$3A_1$&$3A_1$&$3A_1$&$3A_1$         \\
\hline
\end{tabular}
\end{center}

\begin{center}
\begin{tabular}{|c|c|c|c|c|c|c|}
\hline
$22$ $\ast$ & $23$ & $23$         \\
\hline
$H_{10,1}$&$H_{10,1}$&$H_{10,2}$               \\
\hline
$\alpha_{1,1}$,$\alpha_{2,1}$,$\alpha_{1,2}$,$\alpha_{2,2}$&
$\alpha_1$,$\alpha_5$,$\alpha_6$,$\alpha_{24}$&$\alpha_1$,$\alpha_6$             \\
\hline
$A_2$& $3A_1$&$A_1$        \\
\hline
\end{tabular}
\end{center}

\medskip

{\bf n}=10, degeneration $(2\aaa_1)_I$:

\begin{center}
\begin{tabular}{|c|c|c|c|c|c|c|c|c|c|}
\hline
$j$& $18$&$19$ $\ast$  &$21$&$21$&$21$&$21$&$23$&$23$           \\
\hline
$H$&$H_{10,1}$&$H_{10,1}$&$H_{10,1}$&$H_{10,1}$&$H_{10,2}$&$H_{10,2}$&$H_{10,1}$&$H_{10,2}$                \\
\hline
orbit of&$\alpha_{3,1}$,$\alpha_{3,3}$&$\alpha_{3,3}$,$\alpha_{1,4}$&$\alpha_{1,2}$&$\alpha_{2,4}$&
$\alpha_{1,1}$&$\alpha_{2,4}$&$\alpha_{10}$,$\alpha_{15}$&$\alpha_{10}$,$\alpha_{16}$             \\
\hline
$(S^\perp_{N_j})^{(2)}$&$A_3$&$2A_1\oplus A_3$&$A_1\oplus A_3$&$A_1\oplus A_3$&
$2A_1$&$4A_1$&$4A_1$&$2A_1$        \\
\hline
\end{tabular}
\end{center}

\medskip

{\bf n}=10, degeneration $(2\aaa_1)_{II}$:

\begin{center}
\begin{tabular}{|c|c|c|c|c|c|c|}
\hline
$j$&$18$&$21$&$23$ $\ast$           \\
\hline
$H$&$H_{10,1}$&$H_{10,2}$&$H_{10,2}$                \\
\hline
orbit of&$\alpha_{1,5}$&$\alpha_{1,2}$&$\alpha_2$                \\
\hline
$(S^\perp_{N_j})^{(2)}$&$2A_1$&$2A_1$&$2A_1$        \\
\hline
$\natural\{x\in S^\perp_{N_j}|x^2=-4\}$& $80$ & $80$ & $76$ \\
\hline
\end{tabular}
\end{center}

\medskip

{\bf n}=10, degeneration $4\aaa_1$:

\begin{center}
\begin{tabular}{|c|c|c|c|c|c|c|c|c|c|c|}
\hline
$j$ &$18$&$18$&$19$&$19$&$21$&$21$&$21$&$21$            \\
\hline
$H$&$H_{10,1}$&$H_{10,1}$&$H_{10,1}$&$H_{10,1}$&$H_{10,1}$&$H_{10,1}$&
$H_{10,2}$&$H_{10,2}$         \\
\hline
orbit of&$\alpha_{1,1}$,$\alpha_{1,3}$&$\alpha_{2,1}$,$\alpha_{2,3}$&
$\alpha_{2,1}$&$\alpha_{4,1}$&$\alpha_{1,4}$&$\alpha_{2,3}$&
$\alpha_{1,3}$,$\alpha_{3,3}$&$\alpha_{2,3}$             \\
\hline
$(S^\perp_{N_j})^{(2)}$&$A_3$&$A_3$&$2A_3$&$2A_3$&$2A_1\oplus A_3$&$2A_1\oplus A_3$&
$4A_1$&$4A_1$        \\
\hline
\end{tabular}
\end{center}

\begin{center}
\begin{tabular}{|c|c|c|c|c|c|c|}
\hline
$21$&$22$&$23$&$23$ $\ast$         \\
\hline
$H_{10,2}$&$H_{10,1}$&$H_{10,1}$&$H_{10,2}$               \\
\hline
$\alpha_{1,4}$&$\alpha_{1,3}$,$\alpha_{2,3}$,$\alpha_{1,5}$,$\alpha_{2,5}$&
$\alpha_3$,$\alpha_7$&$\alpha_3$,$\alpha_4$,$\alpha_5$,$\alpha_7$\\
\hline
$4A_1$&$2A_2$&$4A_1$&$2A_1$        \\
\hline
\end{tabular}
\end{center}

\medskip

{\bf n}=10, degeneration $8\aaa_1$:

\begin{center}
\begin{tabular}{|c|c|c|c|c|c|c|}
\hline
$j$  &$19$&$21$&$23$ $\ast$           \\
\hline
$H$  &$H_{10,1}$&$H_{10,1}$&$H_{10,1}$                \\
\hline
orbit of&$\alpha_{1,1}$&$\alpha_{1,3}$&$\alpha_2$             \\
\hline
$(S^\perp_{N_j})^{(2)}$&$2A_3$&$2A_1\oplus A_3$&$4A_1$        \\
\hline
\end{tabular}
\end{center}

\medskip

{\bf n}=10, degeneration $2\aaa_2$:

\begin{center}
\begin{tabular}{|c|c|c|c|c|c|c|}
\hline
$j$                    & $22$ $\ast$           \\
\hline
$H$                    &$H_{10,1}$       \\
\hline
orbit of               &$\alpha_{1,9}$   \\
\hline
$(S^\perp_{N_j})^{(2)}$& $2A_2$          \\
\hline
\end{tabular}
\end{center}

\medskip

{\bf n}=12, degeneration $8\aaa_1$:

\begin{center}
\begin{tabular}{|c|c|c|c|c|c|c|}
\hline
$j$                    &  $18$       &    $18$ &$22$&$23$ $\ast$      \\
\hline
$H$                    &$H_{12,1}$   & $H_{12,1}$&$H_{12,1}$&$H_{12,1}$               \\
\hline
orbit of               & $\alpha_{1,1}$ & $\alpha_{2,1}$&$\alpha_{1,1}$,$\alpha_{2,1}$&
$\alpha_1$,$\alpha_4$             \\
\hline
$(S^\perp_{N_j})^{(2)}$& $D_4$        & $D_4$&$3A_1\oplus A_2$&$4A_1$        \\
\hline
\end{tabular}
\end{center}

\medskip

{\bf n}=12, degeneration $\aaa_2$:

\begin{center}
\begin{tabular}{|c|c|c|c|c|c|c|}
\hline
$j$                    & $22$ $\ast$            \\
\hline
$H$                    & $H_{12,1}$                 \\
\hline
orbit of               & $\alpha_{1,5}$,$\alpha_{1,8}$,$\alpha_{1,11}$              \\
\hline
$(S^\perp_{N_j})^{(2)}$&$2A_1\oplus A_2$        \\
\hline
\end{tabular}
\end{center}

\medskip

{\bf n}=16, degeneration $\aaa_1$:

\begin{center}
\begin{tabular}{|c|c|c|c|c|c|c|}
\hline
$j$ & $19$ &$19$&$20$&$20$&$22$&$23$           \\
\hline
$H$ &$H_{16,1}$&$H_{16,1}$&$H_{16,1}$&$H_{16,1}$&$H_{16,1}$&$H_{16,1}$                \\
\hline
orbit of &$\alpha_{1,3}$,$\alpha_{3,3}$,$\alpha_{4,3}$&$\alpha_{2,3}$&
$\alpha_{1,1}$,$\alpha_{4,1}$&$\alpha_{2,1}$,$\alpha_{3,1}$&
$\alpha_{1,1}$,$\alpha_{2,1}$,$\alpha_{1,2}$,$\alpha_{2,2}$&
$\alpha_1$,$\alpha_6$,$\alpha_{19}$,$\alpha_{22}$          \\
\hline
$(S^\perp_{N_j})^{(2)}$&$3A_1$ & $3A_1$&$A_2$&$A_2$&$A_2$&$3A_1$          \\
\hline
\end{tabular}
\end{center}

\medskip

{\bf n}=16, degeneration $5\aaa_1$:

\begin{center}
\begin{tabular}{|c|c|c|c|c|c|c|}
\hline
$j$    &$19$&$19$&$20$&$20$&$22$&$23$ $\ast$  \\
\hline
$H$   &$H_{16,1}$&$H_{16,1}$&$H_{16,1}$&$H_{16,1}$&$H_{16,1}$&$H_{16,1}$               \\
\hline
orbit of&$\alpha_{1,1}$,$\alpha_{3,1}$,$\alpha_{4,1}$&$\alpha_{2,1}$&
$\alpha_{1,2}$,$\alpha_{4,2}$&$\alpha_{2,2}$,$\alpha_{3,2}$&
$\alpha_{1,3}$,$\alpha_{2,3}$,$\alpha_{1,5}$,$\alpha_{2,5}$&
$\alpha_2$,$\alpha_3$,$\alpha_{4}$,$\alpha_{5}$     \\
\hline
$(S^\perp_{N_j})^{(2)}$&$D_4$& $D_4$&$A_4$&$A_4$&$2A_2$&$4A_1$       \\
\hline
\end{tabular}
\end{center}

\medskip

{\bf n}=17, degeneration $\aaa_1$:

\begin{center}
\begin{tabular}{|c|c|c|c|c|c|c|c|c|}
\hline
$j$    &$19$&$19$&$19$&$19$&$21$&$21$&$21$          \\
\hline
$H$ & $H_{17,1}$&$H_{17,2}$&$H_{17,2}$&$H_{17,2}$&$H_{17,1}$&
$H_{17,1}$&$H_{17,2}$                \\
\hline
orbit of&$\alpha_{2,3}$,$\alpha_{2,4}$&$\alpha_{2,3}$&$\alpha_{1,4}$&$\alpha_{2,4}$
&$\alpha_{1,1},\alpha_{3,1}$,$\alpha_{1,5}$,$\alpha_{3,5}$&$\alpha_{2,1}$,$\alpha_{2,5}$&$\alpha_{2,1}$\\
\hline
$(S^\perp_{N_j})^{(2)}$& $A_2$&$D_4$&$3A_1\oplus A_2$&$3A_1\oplus A_2$&$A_1\oplus A_3$&$A_1\oplus A_3$& $A_1$        \\
\hline
\end{tabular}%
\end{center}


\section{Transcendental lattices $T=S^\perp_{L_{K3}}$.}
\label{sec:transclat}

For the fixed type {\bf n} of a finite symplectic automorphism group $G$
and the fixed type of degeneration of codimension one $P$
(Dynkin diagram), a general K\"aherian K3 surface $X$ has the
Picard lattice $S_X\cong S$ described in Theorem \ref{maintheorem}.
Its transcendental lattice $T_X$ is the orthogonal complement
$T_X=(S_X)^\perp_{H^2(X,\bz)}$ where $H^2(X,\bz)$
is an even unimodular lattice of signature $(3,19)$. It is unique
up to isomorphisms, and we denote its isomorphism class as $L_{K3}$.

Thus $T_X\cong T=(S)^\perp_{L_{K3}}$ where $S\subset L_{K3}$ is some
primitive embedding. By epmimorphicity of Torelli map for K3 surfaces,
any such primitive embedding corresponds to K3 surfaces. By Proposition
\ref{prop:primembunim}, the transcendental lattice $T$ can be any lattice
with invariants $(3,19-\rk S, q_T\cong -q_S)$ which are equivalent to
the genus of $T$.

We use the following theorem from \cite{Nik1} which follows from results
by M. Kneser.

\begin{theorem} (\cite[Theorem 1.13.2]{Nik1}) An even lattice $K$ with invariants
$(t_{(+)},t_{(-)}, q)$ is unique if simulataneously

1) $t_{(+)}\ge 1$, $t_{(-)}\ge 1$, $t_{(+)}+t_{(-)}\ge 3$;

2) for each $p\not=2$, either $\rk K\ge l(A_{q_p})+2$, or
$$
q_p\cong q_{\theta_1}^{(p)}(p^k)\oplus q_{\theta_2}^{(p)}(p^k)\oplus q_p^\prime;
$$

3) for $p=2$, either $\rk K\ge l(A_{q_2})+2$, or $q_2\cong u_+^{(2)}(2^k)\oplus q_2^\prime$,
or $q_2\cong v_+^{(2)}(2^k)\oplus q_2^\prime$, or
$$
q_2\cong q_{\theta_1}^{(2)}(2^k)\oplus q_{\theta_2}^{(2)}(2^{k+1})\oplus q_2^\prime.
$$
\label{th:uniqlat}
\end{theorem}

From this Theorem, we then obtain

\begin{theorem}
For a fixed type {\bf n} of a finite symplectic automorphism group $G$
and the fixed type of degeneration of codimension one $P$
(Dynkin diagram), a general K\"aherian K3 surface $X$ has
a unique, up to isomorphisms, transcendental lattice
$T_X\cong T=(S)^\perp_{L_{K3}}$ if $\rk S\le 18$
(by Theorem \ref{th:uniqlat}).

If $\rk S=19$ (equivalently, $\rk T=3$, and then $T$ is positive definite),
then the isomorphism class of the transcendental lattice $T$ is
given in the Table 3 below. The transcendental lattice $T$ is unique
except $({\bf n}=55, 10\aaa_1)$ when there are two possible
isomorphism classes.

Thus, for $({\bf n}=55, 10\aaa_1)$ (equivlaently, $G\cong \AAA_5$ and
the degeneration has the type $10\aaa_1$) there are two non-equivalent
degenerations of codimension one of K\"ahlerian K3 surfaces
which have non-isomorphic transcendental lattices.
\label{th:transuniqe}
\end{theorem}

\begin{proof} The genus of $S$ and then $T$
is calculated in Table 1 above.

If $\rk T\ge 4$, it satisfies Theorem \ref{th:uniqlat}.

If $\rk T = 3$, we calculate $T$ in Table 3 using known tables of
positive definite lattices of the rank $2$ and $3$ for
small determinant (determinants $\le 50$ are enough). See
\cite[Ch. 15, Sects 3, 10]{CS}.
\end{proof}

\begin{table}
\label{table3}
\caption{Transcendental lattices of the rank $3$
of degenerations of codimension $1$ of K\"ahlerian K3 surfaces
with finite symplectic automorphism groups $G=Clos(G)$.}



\begin{tabular}{|c||c|c|c|c|c|c|c|c|c|}
\hline
 {\bf n}& $|G|$& $i$&  $G$ & $Deg$ & $q_T$ &$T$          \\
\hline
\hline
 $26$& $16$& $8$&$SD_{16}$ & $8\aaa_1$& $2_{-1}^{+1},4_{-1}^{+1},16_{-3}^{-1}$ &
 $\left< 2 \right> \oplus \left< 4 \right> \oplus \left< 16 \right>$\\
\hline
     &     &    &        & $2\aaa_2$ & $2_{-5}^{-1},8_{II}^{-2}$ &
$\left<
\begin{array}{ccc}
6 & 2 & -2 \\
2 & 6 &  2 \\
-2& 2 &  6
\end{array}
\right>
$ \\
\hline
\hline
 $32$& $20$& $3$&$Hol(C_5)$&$2\aaa_1$& $4_{-1}^{+1},5^{+3}$ & $A_3(-5)$\\
\hline
     &     &    &          &$5\aaa_1$ &$2_{-1}^{+3},5^{-2}$ &
$\left<
\begin{array}{ccc}
10 & 0 & 0 \\
0  & 4 & 2 \\
0  & 2 & 6
\end{array}
\right>
$ \\
\hline
     &     &    &          &$10\aaa_1$& $4_3^{-1},5^{+2}$ &
$\left<
\begin{array}{ccc}
4 & 2 & 2 \\
2 & 6 &  1 \\
2& 1 &  6
\end{array}
\right>
$ \\
\hline
     &     &    &        &$5\aaa_2$& $2_3^{-1},5^{-2}$ &
$\left<
\begin{array}{ccc}
4 & 1 & -1 \\
1 & 4 &  1 \\
-1& 1 &  4
\end{array}
\right>
$ \\
\hline
\hline
$33$ & $21$ & $1$ & $C_7\rtimes C_3$ & $7\aaa_1$ & $2_{-1}^{+1},7^{+2}$ &
$\left<
\begin{array}{ccc}
14 & 0 & 0 \\
0 &  4 &-1 \\
0& -1  & 2
\end{array}
\right>
$ \\
\hline
\hline
$46$ &$36$ &$9$ &$3^2C_4$ & $6\aaa_1$   & $4_1^{+1},3^{-1},9^{+1}$
& $\left< 36 \right> \oplus A_2(-1)$ \\
\hline
     &     &    &         & $9\aaa_1$   & $2_{-5}^{-3},3^{+2}$ &
$\left< 2 \right>\oplus \left< 6 \right> \oplus \left< 6 \right>$ \\
\hline
     &     &    &        &  $9\aaa_2$   & $2_{-5}^{-1},3^{+2}$ &
$\left< 6 \right>\oplus A_2(-1)$ \\
\hline
\hline
$48$ &$36$ &$10$&$\SSS_{3,3}$&$3\aaa_1$& $2_3^{+3},3^{-2},9^{+1}$ &
$\left< 6 \right>\oplus A_2(-6)$ \\
\hline
     &     &    &        &  $6\aaa_1$& $4_{-1}^{+1},3^{+2},9^{+1}$ &
$\left<
\begin{array}{ccc}
6 & 0 & 3 \\
0 & 6& 3 \\
3 & 3 & 12
\end{array}
\right>
$ \\
\hline
     &     &    &        &  $9\aaa_1$&$2_1^{-3},3^{-3}$
& $3\left< 6 \right>$ \\
\hline
\hline

$51$ & $48$&$48$&$C_2\times \SSS_4$& $2\aaa_1$ & $4_{-1}^{+3},3^{+2}$ &
$\left< 4 \right> \oplus \left< 12 \right>\oplus \left< 12 \right>$ \\
\hline
     &     &    &                  & $4\aaa_1$ & $2_{II}^{+2},8_{-1}^{+1},3^{+2}$ &
$\left<
\begin{array}{ccc}
8 & 2 & -4 \\
2 & 8 &  2 \\
-4& 2 &  8
\end{array}
\right>
$ \\
\hline
     &     &    &                  & $6\aaa_1$   & $4_1^{-3},3^{-1}$ &
$\left< 4 \right>\oplus A_2(-4)$\\
\hline
     &     &    &                  & $8\aaa_1$   & $2_{II}^{-2},4_3^{-1},3^{+2}$ &
$\left< 12 \right>\oplus A_2(-2)$ \\
\hline
     &     &    &                  & $12\aaa_1$  & $2_{II}^{-2},8_1^{+1},3^1$ &
$\left< 8 \right>\oplus A_2(-2)$ \\
\hline
\hline
\end{tabular}
\end{table}

\begin{table}
\begin{tabular}{|c||c|c|c|c|c|c|c|c|}
\hline
 {\bf n}& $|G|$& $i$&  $G$ & $Deg$ & $q_T$ &$T$          \\
\hline
\hline
$55$ & $60$&$5$ &$\AAA_5$ & $\aaa_1$ & $2_1^{-3},3^{-1},5^{-2}$ &
$\left< 2 \right> \oplus A_2(-10)$
\\
\hline
     &     &    &         &  $5\aaa_1$& $2_5^{+3},3^{-1},5^{+1}$ &
$\left<
\begin{array}{ccc}
6 & 0 & 0 \\
0 & 4 &  2 \\
0 & 2 &  6
\end{array}
\right>
$ \\
\hline
     &     &    &        &    $6\aaa_1$&  $4_{-1}^{+1},5^{-2}$ &
$\left<
\begin{array}{ccc}
2 & 1 & 1 \\
1 & 8 & 3 \\
1 & 3 & 8
\end{array}
\right>
$ \\
\hline
     &     &    &        &     $10\aaa_1$ & $4_1^{+1},3^{-1},5^{-1}$ &
$\left<
\begin{array}{ccc}
2 & 0  & 1 \\
0 & 2 &  1 \\
1&  1 &  16
\end{array}
\right>, \
\left<
\begin{array}{ccc}
4 & 0 &  0 \\
0 & 4 &  1 \\
0 & 1 &  4
\end{array}
\right>
$ \\
\hline
     &     &    &        &      $15\aaa_1$ & $2_3^{+3},5^{-1}$ &
$\left< 2 \right>\oplus \left< 2 \right> \oplus \left< 10 \right>$\\
\hline
\hline
$56$ &$64$&$138$&$\Gamma_{25}a_1$&$8\aaa_1$ & $4_4^{-2},8_3^{-1}$ &
$\left< 4 \right>\oplus \left< 4 \right>\oplus \left< 8 \right>$               \\
\hline
     &     &    &        &       $16\aaa_1$ & $4_3^{+3}$& $3\left< 4 \right>$  \\
\hline
\hline
$61$ &$72$ &$43$&$\AAA_{4,3}$   & $3\aaa_1$& $2_3^{-1},4_{II}^{+2},3^{+2}$ &
$\left< 6 \right>\oplus A_2(-4)$\\
\hline
     &     &    &               &$12\aaa_1$& $8_{-1}^{+1},3^{+2}$ &
$\left<
\begin{array}{ccc}
2 & 0 & 1 \\
0 & 6 &  3 \\
1  & 3  &  8
\end{array}
\right>$  \\
\hline
\hline
$65$ &$96$ &$227$&$2^4D_6$      &   $4\aaa_1$&$4_{-3}^{-3},3^{+1}$     &
$\left< 4 \right>\oplus \left< 4 \right>\oplus \left< 12 \right>$  \\
\hline
     &     &     &              &   $8\aaa_1$& $2_{II}^{-2},8_1^{+1},3^{+1}$   &
$\left< 8 \right>\oplus A_2(-2)$\\
\hline
     &     &     &              &   $12\aaa_1$&$4_3^{+3}$ &
$3\left< 4 \right>$\\
\hline
     &     &     &              &   $16\aaa_1$&$2_{II}^{+2},4_5^{-1},3^{+1}$&
$\left< 4 \right>\oplus A_2(-2)$ \\
\hline
\hline
 $75$&$192$&$1023$&$4^2\AAA_4$  &   $16\aaa_1$& $2_{II}^{-2},8_3^{-1}$ &
 $A_3(-2)$               \\
\hline
\hline
\end{tabular}
\end{table}

\newpage

\section{Connected components of moduli of K\"ahlerian K3 surfaces
with a negative definite Picard lattice $M$ and a transcendental lattice $K$}
\label{sec:moduliMK}

Here we apply our results in \cite{Nik0}, \cite{Nik2} about
description of connected components of moduli of K\"ahlerian K3 surfaces
with conditions on Picard lattice. Using these methods and results,
we want to describe connected components of moduli of K\"ahlerian
K3 surfaces $X$ such that the Picard lattice $S_X$ contains
a fixed negative definite primitive sublattice $M$ and the orthogonal
complement $M^\perp_{H^2(X,\bz)}=K$ where the lattice $K$ is also fixed.
A general such K3 surface $X$ has the Picard lattice $S_X=M$ and the
transcendental lattice $T_X=K$.

The lattices $M$ and $K$ are orthogonal complements to each other in the even
unimodular lattice $H^2(X,\bz)$. This defines a canonical isomorphism
$\varphi:q_S\cong -q_K$ which is equivalent to the natural finite index extension
$M\oplus K\subset H^2(X,\bz)$:
$$
H^2(X,\bz)=\{m^\ast \oplus k^\ast\ |\ m^\ast\in M^\ast, k^\ast \in K^\ast,\
\varphi(m^\ast+M)=k^\ast+K\}.
$$
Periods
$$
H^{2,0}(X)+H^{1,1}(X)+H^{0,2}(X)
$$
of $X$ are equivalent to the
positive definite 2-dimensional oriented subspace
$$
\Pi_2(X)=(H^{2,0}(X)+H^{0,2}(X))\cap K\otimes \br\subset K\otimes \br
$$
(its orientation is equivalent to the natural orientation
of the 1-dimensional complex space $H^{2,0}(X)$).
K\"ahler class $c(X)$ of $X$ defines a half $V^+(X)$
of the cone
$$
V(X)=\{x\in H^{1,1}(X)_\br\ |\ x^2>0\}
$$
containing $c(X)$.
Together, they define continuously changing orientations in
all $3$-dimensional positive definite subspaces $\Pi_3\subset H^2(X,\bz)\otimes \br$
and $\Pi_3\subset K\otimes \br$ such that
an oriented basis of $\Pi_2(X)$ together with
$c(X)$ define an oriented basis in $\Pi_2(X)+\br c(X)$.
We denote this orientation as $o(X)\in \{+,-\}$ and the corresponding
$K\otimes \br$ with the choice of such orientation as $(K\otimes \br)_{o(X)}$.
Moreover, effective elements $\delta\in M$ with $\delta^2=-2$
define a fundamental decomposition $P(X):\Delta(M)=
\Delta^+(M)\cup -\Delta^+(M)$
of the set $\Delta(M)$ of roots of $M$ with square $-2$ where
elements of $\Delta^+(M)$ are effective.

We can consider the 4-tuple
$$
(P(X),\varphi(X), o(X),\Pi_2(X))
$$
as periods of a marked K3 surface $X$ with $M\subset S_X$ and
$M^\perp_{H^2(X,\bz)}=K$.

Let $W^{(2)}(M)$ be the group
generated by reflections in elements of $\Delta(M)$. It acts identically
on the discriminant group and the discriminant form $q_M$.
By changing $M\subset S_X$ by $w:M\to M\subset S_X$
where $w\in W^{(2)}(M)$, we can assume that
$P(X)=P$ where the decomposition $P:\Delta(M)=\Delta^+(M)\cup -\Delta^+(M)$
is fixed. This does not changes $\varphi(X),o(X),\Pi_2(X)$.
Let $w_0\in W^{(2)}(M)$ changes the decomposition
$P=\Delta^+(M)\cup -\Delta^+(M)$ to $-P=-\Delta^+(M)\cup \Delta^+(M)$.
Then, by changing $M\subset S_X$ by $-w_0:M\to M\subset S_X$,
$K\subset M^\perp_{H^2(X,\bz)}$ by $-id_K:K\to K\subset M^\perp_{H^2(X,\bz)}$,
if necessary, we can assume that $o(X)=+$ is fixed.
Here, we use that $3$ is odd. Thus, periods of marked in this way
K3 surfaces are given by the pair
$$
(\varphi(X), \Pi_2(X)\subset (K\otimes \br)_+).
$$
It follows that for a fixed isomorphism $\varphi(X)=\varphi:q_M\cong -q_K$ the
spaces of periods and moduli of such K3 surfaces are connected by Global
Torelli Theorem \cite{PS}, \cite{BR} and epimorphicity of period map
\cite{Kul}, \cite{Tod}, \cite{Siu}  for K3 surfaces.

By changing markings to
$g:M\to M\subset S_X$, $f^{-1}:K\to K\subset M^\perp_{H^2(X,\bz)}$
by $g\in O(M)$ with $g(P)=P$ and by $f\in O(K)$, periods
will be changed by equivalent periods
$$
(\overline{f}\varphi \overline{g}, (f\otimes \br) (\Pi_2(X))\subset (K\otimes \br)_{f(+)})
$$
and moduli.

Thus, we obtain

\begin{theorem}
The number of connected components of moduli of K\"ahlerian
K3 surfaces $X$ with Picard lattice $M$ where $M<0$,
and a transcendental lattice isomorphic to $K$ (further we shall call them
as {\bf weak connected components}) is equal to the number of
double cosets
$\overline{O(K)}\backslash O(q)/\overline{O(M)}$ where
$\overline{O(M)}$ and $\overline{O(K)}$ are images of $O(M)$ and $O(K)$
in $O(q)$ where $q=q_M=-q_K$.

Here we consider primitive embeddings $f_1:M\subset S_X$ and $f_2:M\subset S_X$
as equivalent if they are different by an automorphism of
the lattice $M$.
\label{th:moduliMKweak}
\end{theorem}

We remark that the double cosets $\overline{O(K)}\backslash O(q)/\overline{O(M)}$
of the Theorem \ref{th:moduliMKweak} are equivalent to isomorphism
classes of primitive embeddings of the lattice $M$ into $L_{K3}$
with $M^\perp_{L_{K3}}\cong K$
where two such primitive embeddings $f_1:M\subset L_{K3}$,
$f_2:M\subset L_{K3}$ are equivalent if $f_2(M)= h(f_1(M))$ for
$h\in O(L_{K3})$.
See Proposition \ref{prop:primembunim}.
Such isomorphism classes are preserved under continuous
deformations of K3 surfaces since they are discrete data of
the deformations.

Similarly, we obtain

\begin{theorem}
The number of connected components of moduli of K\"ahlerian
K3 surfaces $X$ with Picard lattice $M$ where $M<0$, and fixed
$P(X)=P=\Delta^+(M)\cup -\Delta^+(M)$,
and a transcendental lattice isomorphic to $K$ (further we shall call them as
{\bf strong connected components}) is equal to the number of left cosets
$\overline{O^+(K)}\backslash O(q_K)$ (equivalently, to the index $[O(q_K):\overline{O^+(K)}]$) where
$O^+(K)\subset O(K)$ consists of automorphisms which preserve orientations
$(K\otimes \br)_+$ and $(K\otimes \br)_-$.

Here we consider primitive embeddings $f_1:M\subset S_X$ and $f_2:M\subset S_X$
as equivalent if they are equal.
\label{th:moduliMKstrong}
\end{theorem}

We remark that the left cosets $\overline{O^+(K)}\backslash O(q_K)$
of Theorem \ref{th:moduliMKstrong} are equivalent to all isomorphism
classes of primitive embeddings of the lattice $M$ into $L_{K3}$
with $M^\perp_{L_{K3}}\cong K$, and choices of orientation
$(M^\perp_{L_{K3}}\otimes \br)_{\alpha}$, $\alpha=\pm$
where two such data $f_1:M\subset L_{K3}$, $\alpha_1$
and $f_2:M\subset L_{K3}$, $\alpha_2$ are equivalent if $f_2= hf_1$ and
$h(\alpha_1)=\alpha_2$ for some
$h\in O(L_{K3})$.

In R. Miranda and D.R. Morrison \cite{MM1}, \cite{MM2} (announcement) and D.G. James \cite{J} (proofs),
for indefinite lattices $K$ of Theorem \ref{th:moduliMKstrong}
(equivalently, if $\rk K\ge 4$), the sum
\begin{equation}
e_{-\,-}(K)=\sum_{K^\prime\in g(K)}[O(q_{K^\prime}):\overline{O^+(K^\prime)}]
\label{MMformular}
\end{equation}
is calculated in terms of invariants of the genus $g(K)$ of $K$ (as a particular case of
general results which generalize our \cite[Theorem 1.14.2]{Nik1}). In \cite{MM1}, \cite{MM2}, the group
$O^+(K)$ is denoted as $O_{-\,-}(K)$.  See Theorem in \cite[page 31]{MM2}.

\section{Connected components of moduli of degenerations of codimension one
of K\"ahlerian K3 surfaces with finite symplectic automorphism groups}
\label{sec:moduliST}

Using results of Sec. \ref{sec:moduliMK}, we obtain

\begin{theorem}
For a fixed type {\bf n} of a finite symplectic automorphism group $G$
and a fixed type of degeneration of codimension one $P$
(Dynkin diagram), a general K\"aherian K3 surface $X$ has
a unique (with few exceptions), up to isomorphisms, Picard lattice $S_X\cong S$
described in Theorem \ref{maintheorem} and a transcendental lattice
$T_X\cong T=(S)^\perp_{L_{K3}}$ for some primitive embedding $S\subset L_{K3}$
described in Theorem \ref{th:transuniqe}.

The number $M_w$ of weak connected components of moduli of general
K\"ahelerian K3 surfaces $X$ for such $S$ and $T$ is equal to the number of double cosets
\begin{equation}
M_w=\natural(\overline{O(T)}\backslash O(q)/\overline{O(S)})
\label{doublcos}
\end{equation}
where $q_S=q=-q_T$, in notations of Theorem
\ref{th:moduliMKweak}.
\label{th:moduliweakST}
\end{theorem}

\begin{theorem}
For a fixed type {\bf n} of a finite symplectic automorphism group $G$
and a fixed type of degeneration of codimension one $P$
(Dynkin diagram), a general K\"aherian K3 surface $X$ has
a unique (with few exceptions), up to isomorphisms, Picard lattice $S_X\cong S$
described in Theorem \ref{maintheorem} and a transcendental lattice
$T_X\cong T=(S)^\perp_{L_{K3}}$ for some primitive embedding $S\subset L_{K3}$
described in Theorem \ref{th:transuniqe}.

(1) Assume that $\rk S\le 18$ (equivalently, $\rk T\ge 4$ and then $T$ is
indefinite). Then $T$ is unique, up to isomorphisms, and
the number $M_s$ of strong connected components of moduli of general
K\"ahelerian K3 surfaces $X$ for such $S$ is equal to
$$
M_s=[O(q_T):\overline{O^+(T)}].
$$
For all these cases, $M_s=1$ (and then $M_w=1$) except cases
\begin{equation}
({\bf n},P)=(12,\aaa_2),\ (16,\aaa_1),\ (18,2\aaa_1),\ (22,2\aaa_1),\
(34,2\aaa_1),\ (39,4\aaa_1),\ (40,8\aaa_1),
\label{except1}
\end{equation}
when $M_s=2$. Moreover, for all cases \eqref{except1}, we have
$[O(q_T):\overline{O(T)}]=1$ and then $M_w=1$ (by Theorem \ref{th:moduliweakST}),
except $({\bf n},P)=(16,\aaa_1)$ when $[O(q_T):\overline{O(T)}]=2$.

(2) Assume that $\rk S=19$ (equivalently, $\rk T=3$ and then $T$ is
positive definite) and $({\bf n}, P)\not=(55,10\aaa_1)$.
Then $T$ is unique, up to isomorphisms, and
$$
M_s=[O(q_T):\overline{O^+(T)}]=|O(q_T)|/|O^+(T)/W^+(T)|
$$
where $+$ means "proper" (with determinant=1) automorphisms of $T$,
and $W(T)\subset O(T)$ is generated by reflections in all
roots with square $2$ of $T$. If $({\bf n},P)=(55,10\aaa_1)$, then
$$
M_s=|O(q_{T_1})|/|O^+(T_1)/W^+(T_1)|+|O(q_{T_2})|/|O^+(T_2)/W^+(T_2)|
$$
where $T_1$ and $T_2$ are two non-isomorphic transcendental lattices of this case.

Exact calculations of these invariants are given in Table 4 below.
\label{th:modulistrongST}
\end{theorem}

\begin{proof} In case (1) when $T$ is indefinite and of rank $\rk T\ge 4$, results
follow from Theorem in \cite[page 31]{MM2}. For example, let us consider the case
$({\bf n},P))=(18,2\aaa_1)$. By Table 1, $\rk S=17$ and $q_S\cong 2_{II}^{+2},4_7^{+1},3^{+4}$.
Then $\rk T=22-17=5$ and $q_T=-q_S\cong 2_{II}^{+2},4_1^{+1},3^{+4}$.
In notations of Theorem in \cite[page 31]{MM2}, for $p=2$, we have
 $s(0)=\rk T-l(q_{T_2})=2>0$; for $p=3$, we have
 $\rk T-l(q_{T_3})=5-4=1$, $3\mod 4= 3$, $\det(T)=2^4\cdot 3^4$,
 $\det(K(q_{T_3}))\equiv 3^4\mod (\bz_3^\ast)^2$ and $\Delta =\det(T)/\det (K(q_{T_3}))
 \equiv 2^4\mod (\bz_3^\ast)^2$,  $\left(\frac{\ 2\Delta\ }{\ 3\ }\right)=-1$.
 By Theorem in \cite[page 31]{MM2}, then $e_2=1$, $f_2=4$, $e_3=2$, $f_3=2$,
 type$=(-,-)$, $e_{+\,+}(T)=e_2\cdot e_3=1\cdot 2=2$, the group
 $\widetilde{\Sigma}(T)=\{(+,+),(-,-)\}$, $e_{-\,-}(T)=[O(q_T):\overline{O^+(T)}]=e_{+\,+}(T)=2$
 (the group $O^+(T)$ is denoted as $O_{-\,-}(T)$
 in \cite{MM1}, \cite{MM2}),
 $e(T)=[O(q_T):\overline{O(T)}]=(1/2)e_{+\,+}(T)=1$. Thus, $M_s=2$ and $M_w=1$ in this case.

Let us consider the case (2) when $T$ is positive definite and $\rk T=3$. By \cite[Remark 1.14.6]{Nik1},
the kernel of the natural homomorphism $\pi:O(T)\to O(q_T)$ is equal to $W(T)$ because $\rk T=3<8$.
It follows that the kernel of the natural homomorphism $\pi:O^+(T)\to O(q_T)$ is equal to $W^+(T)$.
Then the order of $|\overline{O^+(T)}|=|\pi(O^+(T))|=|O^+(T)/W^+(T)|$.
It follows the statement.
\end{proof}

We hope to present calculations of missing numbers $M_w$ of weak connected components of moduli
in further variants of the paper and further publications.

\medskip

\begin{table}
\label{table4}
\caption{Automorphism groups of transcendental lattices of rank $3$
of degenerations of codimension $1$ of K\"ahlerian K3 surfaces
with finite symplectic automorphism groups $G=Clos(G)$ and strong connected
components of moduli.}



\begin{tabular}{|c||c|c|c|c|c|c|c|c|c|}
\hline
 {\bf n}& $|G|$& $i$&  $G$ & $Deg$ & $|O(q_T)|$ &$|O(T)|$ &$|W(T)|$ & $M_s$           \\
\hline
\hline
 $26$& $16$& $8$&$SD_{16}$ & $8\aaa_1$& $16$ &$8$ & $2$ & $4$ \\
\hline
     &     &    &        & $2\aaa_2$ & $96$ & $48$ &
$1$ & $4$ \\
\hline
\hline
 $32$& $20$& $3$&$Hol(C_5)$&$2\aaa_1$& $480$ & $48$ & $1$ & $20$ \\
\hline
     &     &    &          &$5\aaa_1$ &$24$ &
$8$ & $1$ & 6 \\
\hline
     &     &    &          &$10\aaa_1$& $16$ & $16$ &
$1$ & $2$ \\
\hline
     &     &    &        &$5\aaa_2$& $12$ & $12$
& $1$ & $2$ \\
\hline
\hline
$33$ & $21$ & $1$ & $C_7\rtimes C_3$ & $7\aaa_1$ & $16$ &
$8$ & $2$ & $4$  \\
\hline
\hline
$46$ &$36$ &$9$ &$3^2C_4$ & $6\aaa_1$   & $24$
&   $24$  & $6$ & $6$ \\
\hline
     &     &    &         & $9\aaa_1$   & $16$ &
$8$ & $2$  & $4$ \\
\hline
     &     &    &        &  $9\aaa_2$   & $8$ &
$24$ & $6$ & $2$  \\
\hline
\hline
$48$ &$36$ &$10$&$\SSS_{3,3}$& $3\aaa_1$&  $432$ &
$24$ & $1$ & $36$ \\
\hline
     &     &    &        &  $6\aaa_1$& $288$ & $16$ &
$1$ & $36$  \\
\hline
     &     &    &        &  $9\aaa_1$&$288$
& $48$ & $1$ & $12$ \\
\hline
\hline

$51$ & $48$&$48$&$C_2\times \SSS_4$& $2\aaa_1$ & $256$ &
$16$ & $1$ & $32$  \\
\hline
     &     &    &                  & $4\aaa_1$ & $64$ & $16$ &
$1$ & $8$ \\
\hline
     &     &    &                  & $6\aaa_1$   & $192$ &
$24$ &  $1$  & $16$ \\
\hline
     &     &    &                  & $8\aaa_1$   & $96$ &
$24$ & $1$ & $8$  \\
\hline
     &     &    &                  & $12\aaa_1$  & $96$ &
$24$ &  $1$ &  $8$  \\
\hline
\hline
$55$ & $60$&$5$ &$\AAA_5$ & $\aaa_1$ & $144$ & $24$
& $2$ & $12$  \\
\hline
     &     &    &         &  $5\aaa_1$& $24$ & $8$ &
$1$ & $6$  \\
\hline
     &     &    &        &    $6\aaa_1$&  $24$ & $8$ &
$2$ & $6$  \\
\hline
     &     &    &        &     $10\aaa_1$ & $8$ & $16$,\
$8$ &  $4$, $1$ & $2$, $2$ \\
\hline
     &     &    &        &      $15\aaa_1$ & $12$ &
$16$ & $4$ &   $3$    \\
\hline
\hline
$56$ &$64$&$138$&$\Gamma_{25}a_1$&$8\aaa_1$ & $64$ &
$16$ & $1$ & $8$                \\
\hline
     &     &    &        &       $16\aaa_1$ & $96$& $48$ & $1$ & $4$  \\
\hline
\hline
$61$ &$72$ &$43$&$\AAA_{4,3}$   & $3\aaa_1$& $96$ &
$24$ & $1$ & $8$ \\
\hline
     &     &    &               &$12\aaa_1$& $16$ & $8$ &
$2$ & $4$ \\
\hline
\end{tabular}
\end{table}

\medskip

\begin{table}
\begin{tabular}{|c||c|c|c|c|c|c|c|c|c|c|}
\hline
 {\bf n}& $|G|$& $i$&  $G$ & $Deg$ & $|O(q_T)|$ &$|O(T)|$ &$|W(T)|$ & $M_s$  \\
\hline
\hline
$65$ &$96$ &$227$&$2^4D_6$      &   $4\aaa_1$&$64$     &
$16$ & $1$ & $8$   \\
\hline
     &     &     &              &   $8\aaa_1$& $96$   &
$24$ &  $1$ &  $8$  \\
\hline
     &     &     &              &   $12\aaa_1$ & $96$ &
$48$ & $1$ & $4$ \\
\hline
     &     &     &              &   $16\aaa_1$&  $24$  &
$24$ &  $1$ & $2$  \\
\hline
\hline
 $75$&$192$&$1023$&$4^2\AAA_4$  &   $16\aaa_1$& $48$ &
 $48$  & $1$ & $2$     \\
\hline
\hline
\end{tabular}
\end{table}

\newpage


\newpage

\section{Appendix: Programs}
\label{sec:Appendix}

Here we give Programs 7 and 8 for GP/PARI Calculator, Version 2.7.0 which were used
for calculations above. They also include Programs 1 - 6 from
\cite{Nik7} --- \cite{Nik9}.

\medskip


Program 7: niemeier$\backslash$genwithorbit.txt

\noindent
$\backslash\backslash$for a Niemeier lattice N\_{}i given by root matrix r\hfill

\noindent
$\backslash\backslash$and cord matrix cord, R=r\^{}-1\hfill

\noindent
$\backslash\backslash$and subgroup H$\backslash$subset A(N\_{}i)\hfill

\noindent
$\backslash\backslash$and its orbits ORB matrix, each line gives\hfill

\noindent
$\backslash\backslash$orbit of length $>$ 1\hfill

\noindent
$\backslash\backslash$it calculates all additional\hfill

\noindent
$\backslash\backslash$1-elements orbits to matrix ORBF and prints it\hfill

\noindent
$\backslash\backslash$(1-elements orbits the last)\hfill

\noindent
$\backslash\backslash$it calculates coinvariant sublattice N\_{}H\hfill

\noindent
$\backslash\backslash$together with morb-orbit given by its first\hfill

\noindent
$\backslash\backslash$element morb of the orbit as SUBLpr below\hfill

\noindent
$\backslash\backslash$by its rational generators,\hfill

\noindent
$\backslash\backslash$and checks if\hfill

\noindent
$\backslash\backslash$it has primitive embedding to L\_{}K3\hfill

\noindent
$\backslash\backslash$Then it prints invariants of its discriminant group DSUBLpr below\hfill

\noindent
$\backslash\backslash$and the intersection matrix rSUBLpr of SUBLpr\hfill

\noindent
sORB=matsize(ORB);m1=0;\hfill

\noindent
for(k1=1,sORB[1],for(k2=1,sORB[2],$\backslash$\hfill

\noindent
if(ORB[k1,k2]==0,,m1=m1+1)));\hfill

\noindent
ORBF=matrix(sORB[1]+(24-m1),sORB[2]);\hfill

\noindent
for(k=1,sORB[1],ORBF[k,]=ORB[k,]);\hfill

\noindent
l=sORB[1];\hfill

\noindent
for(t=1,24,mu=0;for(k1=1,sORB[1],for(k2=1,sORB[2],$\backslash$\hfill

\noindent
if(ORB[k1,k2]!=t,,mu=1)));if(mu==1,,l=l+1;ORBF[l,1]=t));$\backslash$\hfill

\noindent
print("ORBF=",ORBF);$\backslash$\hfill

\noindent
SUBL0=matrix(24,24);alpha=0;$\backslash$\hfill

\noindent
for(k1=1,sORB[1],for(k2=1,sORB[2]-1,$\backslash$\hfill

\noindent
if(ORB[k1,k2+1]$>$0,alpha=alpha+1;$\backslash$\hfill

\noindent
SUBL0[,alpha]=R[,ORB[k1,k2]]-R[,ORB[k1,k2+1]])));$\backslash$\hfill

\noindent
SUBL0[morb,24]=1;$\backslash$\hfill

\noindent
sORBF=matsize(ORBF);$\backslash$\hfill

\noindent
SUBL=SUBL0;$\backslash$\hfill

\noindent
a=matrix(24,24+matsize(cord)[1]);$\backslash$\hfill

\noindent
for(i=1,24,a[i,i]=1);for(i=1,matsize(cord)[1],a[,24+i]=cord[i,]\~{});$\backslash$\noindent

\noindent
L=a;N=SUBL;$\backslash$\noindent

\noindent
ggg=gcd(N);N1=N/ggg;$\backslash$\hfill

\noindent
M=L;$\backslash$\hfill

\noindent
gg=gcd(M);M1=M/gg;$\backslash$\hfill

\noindent
ww=matsnf(M1,1);uu=ww[1];vv=ww[2];dd=ww[3];$\backslash$\hfill

\noindent
mm=matsize(dd)[1];nn=matsize(dd)[2];$\backslash$\hfill

\noindent
nnn=nn;for(i=1,nn,if(dd[,i]==0,nnn=nnn-1));$\backslash$\hfill

\noindent
VV=matrix(nn,nnn);$\backslash$\hfill

\noindent
nnnn=0;for(i=1,nn,if(dd[,i]==0,,nnnn=nnnn+1;VV[,nnnn]=vv[,i]));$\backslash$\hfill

\noindent
M2=M1*VV;MM=M2*gg;$\backslash$\hfill

\noindent
kill(gg);kill(M1);kill(ww);kill(uu);kill(vv);kill(dd);kill(mm);$\backslash$\hfill

\noindent
kill(nn);kill(nnn);kill(nnnn);kill(M2);$\backslash$\hfill

\noindent
L1=MM;kill(VV);$\backslash$\hfill

\noindent
N2=L1\^{}-1*N1;$\backslash$\hfill

\noindent
ww=matsnf(N2,1);uu=ww[1];vv=ww[2];dd=ww[3];$\backslash$\hfill

\noindent
N3=N2*vv;mm=matsize(dd)[1];nn=matsize(dd)[2];$\backslash$\hfill

\noindent
nnn=nn;for(i=1,nn,if(dd[,i]==0,nnn=nnn-1));$\backslash$\hfill

\noindent
N4=matrix(mm,nnn);$\backslash$\hfill

\noindent
nnnn=0;$\backslash$\hfill

\noindent
for(i=1,nn,if(dd[,i]==0,,nnnn=nnnn+1;$\backslash$\hfill

\noindent
ddd=gcd(dd[,i]);N4[,nnnn]=N3[,i]/ddd));$\backslash$\hfill

\noindent
Npr=L1*N4;$\backslash$\hfill

\noindent
kill(ggg);kill(N1);kill(M);kill(L1);kill(MM);$\backslash$\hfill

\noindent
kill(N2);kill(ww);kill(uu);kill(vv);kill(dd);$\backslash$\hfill

\noindent
kill(N3);kill(mm);kill(nn);kill(nnn);kill(nnnn);$\backslash$\hfill

\noindent
kill(ddd);kill(N4);$\backslash$\hfill

\noindent
SUBLpr1=Npr;$\backslash$\hfill

\noindent
R=r;B=SUBLpr1;$\backslash$\hfill

\noindent
l=B\~{}*R*B;$\backslash$\hfill

\noindent
ww=matsnf(l,1);uu=ww[1];vv=ww[2];dd=ww[3];$\backslash$\hfill

\noindent
nn=matsize(l)[1];nnn=nn;for(i=1,nn,if(dd[i,i]==0,nnn=nnn-1));$\backslash$\hfill

\noindent
b=matrix(nn,nnn,X,Y,vv[X,Y+nn-nnn]);$\backslash$\hfill

\noindent
ll=b\~{}*l*b;$\backslash$\hfill

\noindent
d=vector(nnn,X,dd[X+nn-nnn,X+nn-nnn]);$\backslash$\hfill

\noindent
kill(ww);kill(uu);kill(vv);kill(dd);$\backslash$\hfill

\noindent
kill(nn);kill(nnn);$\backslash$\hfill

\noindent
BB=B*b;G=BB\~{}*R*BB;D=d;$\backslash$\hfill

\noindent
SUBLpr=BB;DSUBLpr=D;rSUBLpr=G;

\newpage


Program 8: niemeier$\backslash$genus6.txt

\medskip

\noindent
$\backslash\backslash$for a non-degenerate lattice\hfill

\noindent
$\backslash\backslash$L given by a symmetric integer matrix l\hfill

\noindent
$\backslash\backslash$in some generators\hfill

\noindent
$\backslash\backslash$calculates the elementary difisors (Smyth) basis of L\hfill

\noindent
$\backslash\backslash$as a matrix b and\hfill

\noindent
$\backslash\backslash$calculates the matrix ll=b\~{}*l*b

\noindent
$\backslash\backslash$of L in the bases b\hfill

\noindent
$\backslash\backslash$calculates invariants d of L$\backslash$subset L\^{}$\backslash$ast\hfill

\noindent
$\backslash\backslash$for primes, p, calculates llll=L$\backslash$otimes $\backslash$bz\_{}p\hfill

\noindent
$\backslash\backslash$in Smith form\hfill

\noindent
$\backslash\backslash$thus, calculates genus of L\hfill


\noindent
ww=matsnf(l,1);uu=ww[1];vv=ww[2];dd=ww[3];\hfill

\noindent
nn=matsize(l)[1];nnn=nn;for(i=1,nn,if(dd[i,i]==0,nnn=nnn-1));\hfill

\noindent
b=matrix(nn,nnn,X,Y,vv[X,Y+nn-nnn]);\hfill

\noindent
ll=b\~{}*l*b;\hfill

\noindent
d=vector(nnn,X,dd[X+nn-nnn,X+nn-nnn]);\hfill

\noindent
kill(ww);kill(uu);kill(vv);kill(dd);\hfill

\noindent
kill(nn);kill(nnn);\hfill

\noindent
n=matsize(d)[2];\hfill

\noindent
delta=vector(n,X,d[n+1-X]);\hfill

\noindent
bb=matrix(n,n,X,Y,b[X,n+1-Y]);\hfill

\noindent
lll=bb\~{}*l*bb;\hfill

\noindent
F=factor(delta[n]);\hfill

\noindent
f1=matsize(F)[1];\hfill

\noindent
for(KK1=1,f1,p=F[KK1,1];t=F[KK1,2];$\backslash$\hfill

\noindent
u=vector(n);$\backslash$\hfill

\noindent
for(k2=1,n,u[k2]=gcd(delta[k2],p\^{}t));$\backslash$\hfill

\noindent
v=vector(n);j1=1;$\backslash$\hfill

\noindent
v[j1]=1;for(k2=2,n,if(u[k2]$>$u[k2-1],j1=j1+1;v[j1]=k2,));$\backslash$\hfill

\noindent
vv=vector(j1,X,v[X]);$\backslash$\hfill

\noindent
nvv=matsize(vv)[2];$\backslash$\hfill

\noindent
for(k4=1,nvv,if(k4$<$nvv,ss=vv[k4+1]-vv[k4];$\backslash$\hfill

\noindent
cc=matrix(ss,ss,X,Y,lll[vv[k4]+X-1,vv[k4]+Y-1]/u[vv[k4]]);$\backslash$\hfill

\noindent
ccdet=matdet(cc);dcc=Mod(ccdet,p);kron=kronecker(ccdet,p);$\backslash$\hfill

\noindent
if(p$>$2,print(u[v[k4]]," size=",ss," det=",dcc," kro=",kron),),$\backslash$\hfill

\noindent
ss=n-v[k4]+1;cc=matrix(ss,ss,X,Y,lll[vv[k4]+X-1,vv[k4]+Y-1]/u[vv[k4]]);$\backslash$\hfill

\noindent
ccdet=matdet(cc);dcc=Mod(ccdet,p);kron=kronecker(ccdet,p);$\backslash$\hfill

\noindent
if(p$>$2,print(u[v[k4]]," size=",ss," det=",dcc," kro=",kron),)));$\backslash$\hfill

\noindent
if(p!=2,,llll=lll;for(k1=1,nvv,$\backslash$\hfill

\noindent
if(k1$<$nvv,ss1=vv[k1+1]-vv[k1],ss1=n-vv[k1]+1);$\backslash$\hfill

\noindent
cc1=matrix(ss1,ss1,X,Y,llll[vv[k1]+X-1,vv[k1]+Y-1]);$\backslash$\hfill

\noindent
ty=0;for(k=1,ss1,if(Mod(cc1[k,k]/u[vv[k1]],2)==Mod(0,2),,ty=1));$\backslash$\hfill

\noindent
dcc1=Mod(matdet(cc1/u[vv[k1]]),8);$\backslash$\hfill

\noindent
a=cc1/u[vv[k1]];$\backslash$\hfill

\noindent
na=matsize(a)[1];$\backslash$\hfill

\noindent
alpha==1;$\backslash$\hfill

\noindent
while(alpha,$\backslash$\hfill

\noindent
t=0;beta=0;mu=1;$\backslash$\hfill

\noindent
for(k=1,na,if(Mod(a[k,k],2)==Mod(1,2)\&\&mu=1,t=k,mu=0));$\backslash$\hfill

\noindent
for(k=1,t,for(k1=k+1,na,if(a[k,k1]==0,,beta=1)));$\backslash$\hfill

\noindent
for(k=t+1,na,if(Mod(a[k,k],2)==Mod(0,2),,beta=1));$\backslash$\hfill

\noindent
if(beta==1,,alpha=0);$\backslash$\hfill

\noindent
t1=0;gam1=1;for(k=1,na,if(Mod(a[k,k],2)==Mod(1,2)\&\&gam1==1,$\backslash$\hfill

\noindent
for(k1=k+1,na,if(a[k,k1]==0\&\&gam1==1,,gam1=0)),gam1=0);$\backslash$\hfill

\noindent
if(gam1==1,t1=t1+1,));$\backslash$\hfill

\noindent
t2=t1;gam2=1;$\backslash$\hfill

\noindent
for(k=t1+1,na,if(Mod(a[k,k],2)==Mod(1,2)\&\&gam2==1,t2=k;gam2=0,));$\backslash$\hfill

\noindent
if(t2==t1,,$\backslash$\hfill

\noindent
if(t2==t1+1,,trans=matrix(na,na);for(k=1,na,trans[k,k]=1);$\backslash$\hfill

\noindent
trans[t2,t1+1]=1;trans[t1+1,t1+1]=0;trans[t1+1,t2]=1;trans[t2,t2]=0;$\backslash$\hfill

\noindent
a=trans\~{}*a*trans);$\backslash$\hfill

\noindent
trans=matrix(na,na);for(k=1,na,trans[k,k]=1);$\backslash$\hfill

\noindent
for(m=t1+2,na,trans[t1+1,m]=-a[t1+1,m]/a[t1+1,t1+1]);$\backslash$\hfill

\noindent
a=trans\~{}*a*trans));$\backslash$\hfill

\noindent
sign8=Mod(0,8);$\backslash$\hfill

\noindent
for(k=1,t,sign8=sign8+Mod(a[k,k],8));$\backslash$\hfill

\noindent
print(u[vv[k1]]," size=",ss1," type=",ty," det=",dcc1," sign8=",sign8);$\backslash$\hfill

\noindent
kill(alpha);kill(sign8);kill(t);kill(t1);kill(t2);kill(a);kill(trans);$\backslash$\hfill

\noindent
kill(beta);kill(mu);kill(gam1);kill(gam2);kill(na);$\backslash$\hfill

\noindent
ccc1=cc1\^{}-1;$\backslash$\hfill

\noindent
for(k2=k1+1,nvv,if(k2$<$nvv,ss2=vv[k2+1]-vv[k2],ss2=n-vv[k2]+1);$\backslash$\hfill

\noindent
cc21=matrix(ss1,ss2,X,Y,llll[vv[k1]+X-1,vv[k2]+Y-1]);$\backslash$\hfill

\noindent
cc21n=ccc1*cc21;ttt=matrix(n,n);for(aaa1=1,n,ttt[aaa1,aaa1]=1);$\backslash$\hfill

\noindent
for(aa1=1,ss1,for(aa2=1,ss2,ttt[vv[k1]+aa1-1,vv[k2]+aa2-1]=-cc21n[aa1,aa2]));$\backslash$\hfill

\noindent
llll=ttt\~{}*llll*ttt))));\hfill


V.V. Nikulin
\par Steklov Mathematical Institute,
\par ul. Gubkina 8, Moscow 117966, GSP-1, Russia;

\vskip5pt

Deptm. of Pure Mathem. The University of
Liverpool, Liverpool\par L69 3BX, UK
\par

\vskip5pt

nikulin@mi.ras.ru\, \ \
vnikulin@liv.ac.uk \, \ \  vvnikulin@list.ru

Personal page: http://vnikulin.com

\end{document}